\documentclass[a4paper,10pt]{article}
\usepackage{graphicx}
\usepackage{subfigure}
\usepackage{epstopdf}
\usepackage{booktabs}
\usepackage{threeparttable}
\usepackage{amssymb, amsmath, amsthm, epsfig}
\numberwithin{equation}{section}  
\numberwithin{figure}{section}    
\numberwithin{table}{section}     
\usepackage{latexsym, enumerate}
\usepackage{eepic}
\usepackage{epic}
\usepackage{color}

\usepackage{amsmath,textcomp, amssymb, mathrsfs, bm, amsthm, amsfonts}
\usepackage{algpseudocode}
\allowdisplaybreaks[4]
\usepackage{ifpdf}
\usepackage{siunitx}%
\usepackage{chngcntr}

\usepackage{dsfont}
\usepackage{multirow}
\usepackage{bm}
\usepackage{caption}
\usepackage[colorlinks, linkcolor=blue, citecolor=blue]{hyperref}

\allowdisplaybreaks[4]

\theoremstyle{plain}
\newtheorem{lem}{Lemma}[section]
\newtheorem{thm}[lem]{Theorem}
\newtheorem{cor}[lem]{Corollary}

\theoremstyle{definition}
\newtheorem{defn}{Definition}[section]

\theoremstyle{remark}
\newtheorem{rem}{Remark}[section]

\renewcommand{\thefigure}{\thesection.\arabic{figure}}

\def \d{\ensuremath{\mathrm{d}}}
\begin{document}
\renewcommand{\figurename}{Figure}
\renewcommand{\thesubfigure}{(\alph{subfigure})}
\renewcommand{\thesubtable}{(\alph{subtable})}
\makeatletter
\renewcommand{\p@subfigure}{\thefigure~}

\makeatother
\title{\large\bf Hybridization theory for plasmon resonance in metallic nanostructures}
\author{
Qi Lei\thanks
{School of Mathematics, Hunan University, Changsha 410082, Hunan Province, China.
Email: leiqi@hnu.edu.cn}
\and
Hongyu Liu\thanks
{Department of Mathematics, City University of Hong Kong, Kowloon, Hong Kong, China.
Email:  hongyu.liuip@gmail.com; hongyliu@cityu.edu.hk}
\and
Zhi-Qiang Miao\thanks
{School of Mathematics and Statistics, Central South University, Changsha 410083, Hunan Province, China, and
Department of Mathematics, City University of Hong Kong, Kowloon, Hong Kong, China.
Email: zhiqmiao@csu.edu.cn; zhiqmiao@cityu.edu.hk}
\and
Guang-Hui Zheng\thanks
{School of Mathematics, Hunan University, Changsha 410082, Hunan Province, China.
Email: zhenggh2012@hnu.edu.cn; zhgh1980@163.com}
}
\date{}
\maketitle

\begin{center}{\bf ABSTRACT}
\end{center}\smallskip
In this paper, we investigate the hybridization theory of plasmon resonance in metallic nanostructures, which has been validated by the authors in \cite{prodan2003hybridization} through a series of experiments. In an electrostatic field, we establish a mathematical framework for the Neumann-Poincar\'{e}(NP) type operators for metallic nanoparticles with general geometries related to core and shell scales. We calculate the plasmon resonance frequency of concentric disk metal nanoshells with normal perturbations at the interfaces by the asymptotic analysis and perturbation theory to reveal the intrinsic hybridization between solid and cavity plasmon modes. The theoretical finding are convincingly supported by extensive numerical experiments. Our theory corroborates and strengthens that by properly enriching the materials structures as well as the underlying geometries, one can induce much richer plasmon resonance phenomena of practical significance.

\smallskip
{\bf keywords}: hybridization theory; plasmon resonance; Neumann-Poincar\'{e} type operator; perturbation

\section{Introduction}

    When an incident wave of a particular frequency strikes a nanoparticle, the incident wave pushes free electrons onto the metal surface. The Coulomb force strengthens the strong field inside and around the nanoparticle, which triggers the surface plasmon resonance.
    It is possible to identify surface enhancement effects and obtain high-resolution spectroscopic imaging by combining near-field microscopy with nanostructures with plasmon resonance capabilities.
    Surface-enhanced Raman spectroscopy (SERS), in which the molecules to be measured are adsorbed on the surface of metallic nanostructures, is the most common application of plasmon resonance, which is also crucial for molecular identification. When laser light is shone onto these nanostructures, the plasmon resonance creates a local electric field enhancement effect that amplifies the molecules' Raman scattering signals.
    The SERS approach can be used to identify and detect compounds with great sensitivity because each molecule has a unique Raman scattering spectrum.  Furthermore, photothermal therapy, optical filters, medical imaging, and electromagnetic shielding frequently employ the plasmonic resonance effect of metal nanoshells \cite{ahmadi2014potential,anker2008biosensing,baffou2010mapping,beier2007application,hirsch2006metal,jain2006calculated,loo2005immunotargeted,reed1972methods}.

    Grieser was the first to provide a rigorous mathematical framework for plasmon by offering some more sophisticated mathematical approaches to explore the plasmonic eigenvalues and situating the metallic particle in a natural mathematical setting \cite{grieser2014plasmonic}.
    Ammari and his colleagues were the first to accurately define the notion of plasmonic resonance mathematically and to use the Helmholtz equation to model the propagation of electromagnetic waves. They examined the enhanced scattering and absorption characteristics of nanoparticles exhibiting plasmon resonance, as well as how plasmon effects can be used to improve super-resolution and super-focusing in optical applications. They also methodically examined the effects of plasmon resonance in response to changes in the size and shape of nanoparticles \cite{ammari2017mathematical}. Subsequently, they then examined the plasmon resonance of nanoparticles using the full Maxwell equations governing light propagation. They created an effective medium theory specifically designed for resonant plasmonic systems by utilizing the layer potential technique. Furthermore, they also determined the Maxwell-Garnett theory's effective volume fraction conditions within the setting of plasmon resonance, which helped to clarify the relationship between resonant behavior and nanoparticle characteristics  \cite{ammari2016mathematicals}.
    In \cite{ando2016analysis,ando2016plasmon}, the authors investigated plasmon resonance for the finite frequency Helmholtz equation and the smooth domain conductivity equation. Meanwhile, considerable advancements have been made in the mathematical theory and framework of plasmon resonance in practical applications. Using layer potential and asymptotic analysis approaches, the authors of \cite{ammari2018heat} address the heat produced by nanoparticles under plasmon resonance irradiation and offer a solution to the problem of tracking the rise in tissue temperature brought on by nanoparticle heating in biological media. In a finite frequency range beyond the quasi-static approximation, the authors of \cite{li2018anomalous,li2019analysis} investigated time-harmonic electromagnetic scattering from plasmonic inclusions, as well as anomalous local resonances and cloaking. In addition,  the development of mathematical analyses of the plasmonic resonances of metallic nanoparticles with a variety of geometrical structures, including bowtie and slit types, confocal ellipsoids, infinitely thick metal plates with gap defects, anisotropic straight and curved nanorods, and photonic crystals, has been made possible by the recent rapid advancements in the chemical synthesis of metallic nanostructures \cite{bonnetier2018plasmonic,chung2014cloaking,deng2022plasmon,deng2021mathematical,lin2019integral,lin2019mathematical,zheng2019mathematical}.

    As research progressed, the authors in
    \cite{deng2024elastostatics,fang2023plasmon}
    considered the plasmon resonance in multi-layer structure and established a mathematical framework in electrostatic fields and elastic systems, while their multi-layer structure is formed by a special geometry of concentric balls. We conclude that strategically designed, more complex metal nanoparticles can produce stronger and broader plasmonic resonances, while also allowing precise tuning of the optical effects by adjusting the shape. Therefore, in this study, we investigate a general two-dimensional bilayer metal nanoshell structure, based on the evolving chemical synthesis techniques for nanostructures.
    We apply distinct small normal perturbations to the concentric disk core and shell interfaces, respectively, to make the structure of the plasmon more general, which will be crucial for our future probing of the plasmon resonance in metal nanoshells with arbitrary geometries.

    We highlight a noteworthy finding by Hirsch in his research on spherical metal nanoshell particles: the plasmon resonance wavelength of the particles redshifts as the nanoshell thickness decreases \cite{hirsch2006metal}. We have concentrated numerically on this perturbation of metal nanoparticles in the antibonding and bonding plasmon modes as influenced by the core and shell scale factors, as well as the two limiting cases of solid and cavity plasmon modes, because we are highly interested in the tunable property of metal nanoshells. We link the above plasmon modes using hybridization theory and provide explicit mathematical calculations. First, we use the transmission condition to construct the NP-type operator implicit in the problem, each of whose eigenvalues corresponds to a resonant mode. Next, we use the NP operator's scale transformation property to extract the scale factor. We obtain the perturbed eigenvalues using standard perturbation theory and asymptotic analysis, allowing the plasmonic resonant frequencies to be analyzed. A general design principle that can be used both qualitatively and quantitatively to direct the design of metal nanoshell structures and forecast their resonance properties is provided to nano researchers by the mathematical analysis of plasmonic resonant frequencies influenced by the core and shell scales established in this paper.

    The paper is organized in the following way. The key results about the plasmonic resonant frequencies and the mathematical settings are presented in Section \ref{sec:math-setting}. Basic results about NP-type operators and layer potentials are given in section \ref{sec:pre}. The purpose of Section \ref{sec:pertur} is to demonstrate the primary finding covered in Section \ref{sec:math-setting}. We provide numerical examples in Section \ref{sec:num} to support our theoretical conclusions.
\section{Mathematical setting and main results}\label{sec:math-setting}
    The rapid development in the field of nanomaterials has brought great advances in science, engineering and technology. It is particularly important to use mathematical tools to establish the expression of plasmon resonance in metallic nanoshell particles, which is closely related to the practical application of plasmon resonance, and to provide an understanding of how it can be utilized as a nano-optical element to control light at the nanoscale.

    \begin{figure}[htpb]
    	\centering    	\includegraphics[width=0.8\linewidth]{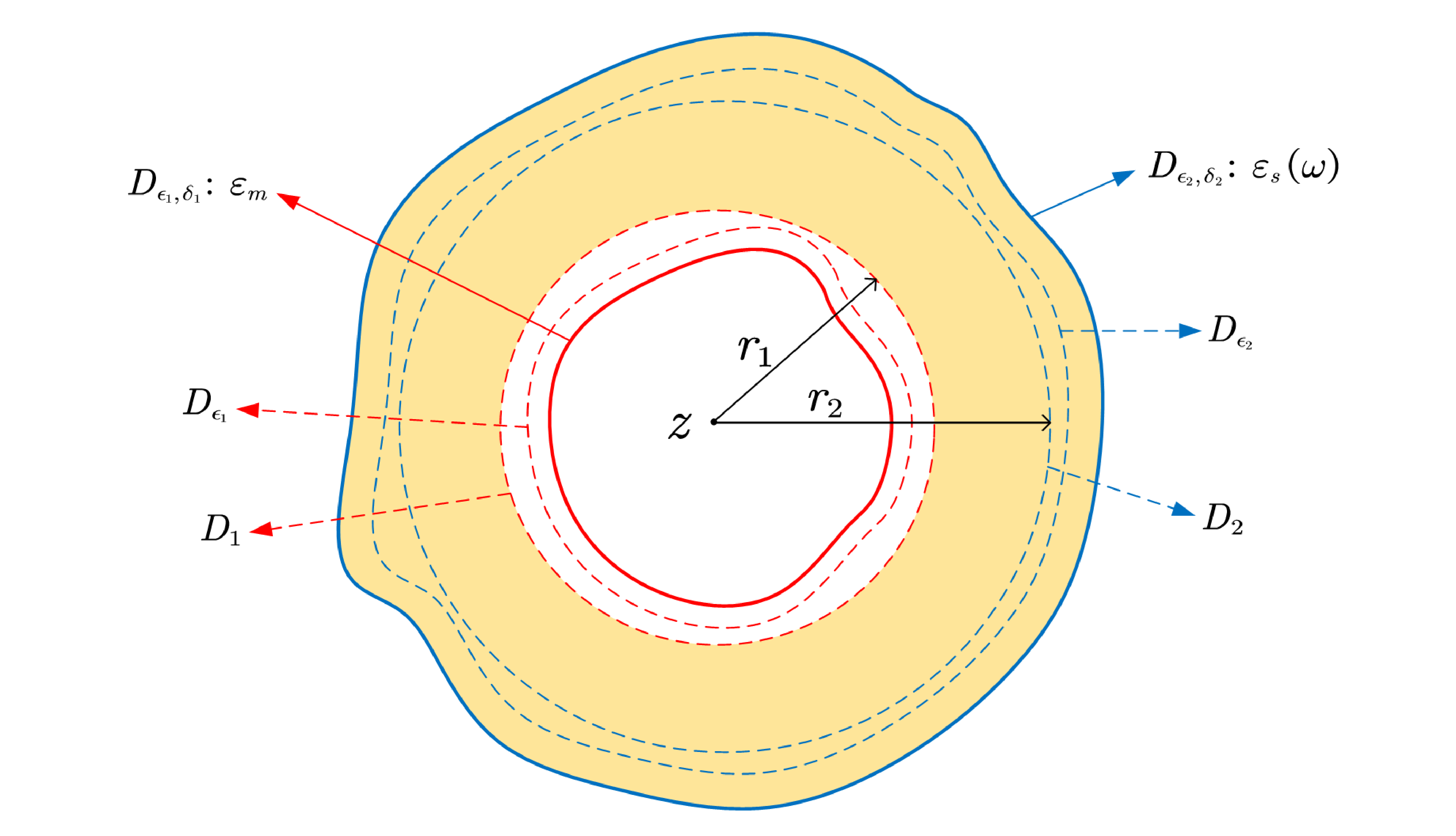}
    	\caption{Schematic illustration of metallic nanostructures with perturbed core-shell boundaries.} 
    	\label{picture-pertubation}
    \end{figure}

    In order to accurately describe the plasmon resonance of concentric disk interface perturbations, for $i=1,2$, the bounded domain $D_i \subset \mathbb{R} ^2 $ denotes the disk with radius $r_i$, interface $\partial D_i$. The $\epsilon _i$-normal perturbation of $D_i$ yields the domain $D_{\epsilon _i}$, whose interface $\partial D_{\epsilon _i}$ can be defined as
    \begin{equation*}
    	\partial D_{\epsilon _i}=\left\{ \tilde{x}~|~\tilde{x}= x +\epsilon _ih_i\left( x \right) \nu _i\left( x \right), x \in \partial D_i  \right\} ,
    \end{equation*}
    where the function $h_i\in C^{2,\beta}\left( \partial D_i \right) $. Next, the two perturbed bounded domains $D_{\epsilon _i}$ are scaled with different proportions $\delta _i$ to obtain the domains $D_{\epsilon _i, \delta _i}$ with interface $\partial D_{\epsilon _i, \delta _i} = \delta _i \cdot \partial D_{\epsilon _i} $. The bounded domain $D_{\epsilon _1, \delta _1}$ denotes the dielectric core with permittivity $\varepsilon _m = 1$, and the bounded domain $D_{\epsilon _2, \delta _2}\setminus \overline{D_{\epsilon _1,\delta _1}}$ denotes a noble metal shell with dielectric function $\varepsilon _s \left( \omega \right)$ expressed by the Drude's model
    \begin{equation*}
    	\label{drude3}
        \varepsilon\left( \omega \right) = 1 - \frac{\omega_p^2}{\omega^2 + i\gamma\omega},
    \end{equation*}
    where $\omega_p$ represents the bulk plasmon frequency, $\gamma$ is the metallic loss factor for absorption and this work considers lossless noble metal material. The region $\mathbb{R} ^2\setminus \overline{D_{\epsilon _2,\delta _2}}$ is filled with a solid dielectric with the same permittivity $\varepsilon _m = 1$. Assuming that $\left| D_{i} \right|=O\left( 1 \right)$, then the scaling factors $\delta _1$, $\delta _2$ carve the scaling variations of the metal nanoshell core and shell respectively.  The metal nanoshell geometry we are interested in is illustrated in Figure \ref{picture-pertubation}.
     The transmission conditions along the $\partial D_{\epsilon _1,\delta _1}$ and $\partial D_{\epsilon _2,\delta _2}$ lead to the total field $u$ satisfying the following electrostatic field model.

    \begin{equation}
    	\label{perturbations_math_model}
    \left\{
    	\begin{aligned}
    		&\nabla \cdot \left( \varepsilon \nabla u\left( x \right) \right) =0, \quad &&x\in \mathbb{R} ^2, \vspace{1ex} \\
    		&\left. u\left( x \right) \right|_+=\left. u\left( x \right) \right|_- ,\quad &&x\in \partial D_{\epsilon _1, \delta _1}\cup \partial D_{\epsilon _2,\delta _2}, \vspace{1ex} \\
    		&\varepsilon _s\left( \omega \right) \left. \dfrac{\partial u}{\partial \nu}\left( x \right) \right|_+=\varepsilon _m\left. \dfrac{\partial u}{\partial \nu}\left( x \right) \right|_-, \quad  &&x\in \partial D_{\epsilon _1,\delta _1}, \vspace{1ex} \\
    		&\varepsilon _m\left. \dfrac{\partial u}{\partial \nu}\left( x \right) \right|_+=\varepsilon _s\left( \omega \right) \left. \dfrac{\partial u}{\partial \nu}\left( x \right) \right|_-, \quad  &&x\in \partial D_{\epsilon _2,\delta _2}, \vspace{1ex} \\
    		&u\left( x \right) -H\left( x \right) =O\left( |x|^{-1} \right),  \quad &&|x|\rightarrow \infty,\\
    	\end{aligned}
    \right.
    \end{equation}
    where
    \begin{equation*}
    	\varepsilon =\varepsilon _m\chi \left( D_{\epsilon _1,\delta _1}\cup \mathbb{R} ^2\setminus \overline{D_{\epsilon _2,\delta _2}} \right) +\varepsilon _s\left( \omega \right) \chi \left( D_{\epsilon _2,\delta _2}\setminus \overline{D_{\epsilon _1,\delta _1}} \right),
    \end{equation*}
    and the harmonic function $H\left( x \right)$ represents an incident field. The solution $u \left( x \right)$ to (\ref{perturbations_math_model}) may be represented as \cite{ammari2013spectral}
    \begin{equation*}
    	u\left( x \right) =H\left( x \right) +\mathcal{S} _{D_{\epsilon _1,\delta _1}}\left[ \bar{\varphi} \right] \left( x \right) +\mathcal{S} _{D_{\epsilon _2,\delta _2}}\left[ \bar{\phi} \right] \left( x \right),
    \end{equation*}
    for some functions  $\bar{\varphi}\left( x \right) \in L_{0}^{2}\left( \partial D_{\epsilon _1,\delta _1} \right)$ and $\bar{\phi}\left( x \right) \in L_{0}^{2}\left( \partial D_{\epsilon _2,\delta _2} \right) $, $L_{0}^{2}$ denotes the set of all square integrable functions with the integral 0.  The main results of this paper are given in the following theorems and remarks. The proofs will be given in section \ref{sec:pertur}.

    \begin{thm}
    \label{theorem2.1}
    Let the core-shell interfaces $\partial D_i$ of the concentric disk-shaped metallic nanoshell exhibit the following $\epsilon _i$-normal perturbation
     \begin{equation*}
    	\partial D_{\epsilon _i}=\left\{ \tilde{x}~|~\tilde{x}= x +\epsilon _i h_i\left( x \right) \nu _i\left( x \right), x \in \partial D_i \right\} ,
    \end{equation*}
     then, the first-order corrected resonance frequency $\tilde{\omega} _{n\pm}\left( \delta _1 , \delta _2 \right)$ with respect to the scale factors $\delta _1$, $\delta _2$ is given by the following expression
    \begin{equation}
        \label{2d_perturbation_resonance_frequency}
    	\tilde{\omega} _{n\pm}\left( \delta _1 , \delta _2 \right) =\frac{\omega _{p}}{\sqrt{2}}\sqrt{1-2 \lambda _n}
    \left[1-\left( 1-2 \lambda _n \right) ^{-1}\left( \epsilon _1\left< \mathbb{K} _{\delta _1 , \delta _2 , 1}^{*} \varPhi _n,\varPhi _n \right> _{\mathcal{H} ^2} + \epsilon _2\left< \mathbb{K} _{\delta _1 , \delta _2 , 2}^{*} \varPhi _n,\varPhi _n \right> _{\mathcal{H} ^2} \right)\right] , \  n=1,2,\cdots,
    \end{equation}
    where the inner product $\left<  \cdot, \cdot \right> _{\mathcal{H} ^2}$ is presented in \textnormal{(\ref{H^2_product})}, $\left( \lambda _n, \varPhi _n \right)$ is the eigenpair of the operator $\mathbb{K} _{\delta _1 , \delta _2 }^{*}$ in \textnormal{(\ref{three-matrix-1})}, and $\mathbb{K} _{\delta _1 , \delta _2 , 1}^{*}$ and $\mathbb{K} _{\delta _1 , \delta _2 , 2}^{*}$ are presented in \textnormal{(\ref{three-matrix-2})} .
    \end{thm}
    \begin{rem}
    Two limit states are important in the above results. The scale factor $\delta_1 \rightarrow 0$, the structure is called the solid plasmon, while the scale factor $\delta_2 \rightarrow \infty$, the structure is called the cavity plasmon. It is easy to calculate the resonance frequency $\tilde{\omega} _{s,n}\left( \delta _2 \right)$ and $\tilde{\omega} _{c,n}\left( \delta _1 \right)$ of the perturbed metallic nanoshell in the solid plasmon and cavity plasmon are as follows
    \begin{align*}
    	\tilde{\omega} _{s,n}\left( \delta _2 \right) & =\frac{\omega _{p}}{\sqrt{2}}\left[1-\epsilon _1\left< \mathbb{K} _{0,\delta_2 , 1}^{*} \varPhi _n,\varPhi _n \right> _{\mathcal{H} ^2} - \epsilon _2\left< \mathbb{K} _{0,\delta_2 , 2}^{*} \varPhi _n,\varPhi _n \right> _{\mathcal{H} ^2}\right] ,\\
    	\tilde{\omega} _{c,n}\left( \delta _1 \right) & =\frac{\omega _{p}}{\sqrt{2}}\left[1-\epsilon _1\left< \mathbb{K} _{\delta_1,\infty , 1}^{*} \varPhi _n,\varPhi _n \right> _{\mathcal{H} ^2} - \epsilon _2\left< \mathbb{K} _{\delta_1,\infty , 2}^{*} \varPhi _n,\varPhi _n \right> _{\mathcal{H} ^2}\right] ,
    \end{align*}
    these two frequencies satisfy the following asymptotic sum rules \cite{apell1996sum}
    \begin{equation*}
        \tilde{\omega} _{s,n}^2\left( \delta _2 \right) + \tilde{\omega} _{c,n}^2\left( \delta _1 \right) = \omega _{p}^2+\mathcal {O}\left( \epsilon _1 \right)+\mathcal {O}\left( \epsilon _2 \right), \quad  n=1,2,\cdots.
    \end{equation*}

    In particular, as $\epsilon _1 \rightarrow 0 $ and $\epsilon _2 \rightarrow 0 $, the geometric structure degenerates into a concentric disk, and a straightforward calculation yields
    \begin{equation*}
    	\omega _{s}= \frac{\omega _{p}}{\sqrt{2}}, \quad \omega _{c} = \frac{\omega _{p}}{\sqrt{2}},
    \end{equation*}
    where $\omega _{s}$ and $\omega _{c}$ are the resonant frequencies of solid sphere plasmon and cavity plasmon, respectively, satisfying the sum rules
    \begin{equation}
        \label{omega-p^2}
    	\omega _{s}^2 + \omega _{c}^2 = \omega _{p}^2.
    \end{equation}
    \end{rem}
    \begin{rem}
    In (\ref{2d_perturbation_resonance_frequency}), by adjusting the shape function $h_i$, we can determine the resonant frequencies for a broader range of geometries, extending beyond simple configurations. Furthermore, by simultaneously tuning $\delta_i$ and $\epsilon _i$, it is possible to achieve precise control over the resonance frequencies, enabling the design of structures with tailored optical properties. This flexibility allows for a wider bandwidth of resonance frequencies, which is particularly beneficial in applications such as sensing technologies and molecular detection.
    \end{rem}
    \begin{rem}
    By discretizing the layer potential operators $\mathbb{K} _{\delta _1 , \delta _2 , 1}^{*}$ and $\mathbb{K} _{\delta _1 , \delta _2 , 2}^{*}$ as introduced in (\ref{2d_perturbation_resonance_frequency}), we can obtain numerical results for resonance frequencies of more general shapes \cite{colton2019inverse}. This approach allows for a more comprehensive analysis of the plasmon resonance in arbitrarily shaped nanostructures, which is crucial for understanding their optical properties and potential applications in fields such as nanophotonics and metamaterials.
    \end{rem}

    \begin{figure}[htpb]
    	\centering    	\includegraphics[width=0.8\linewidth]{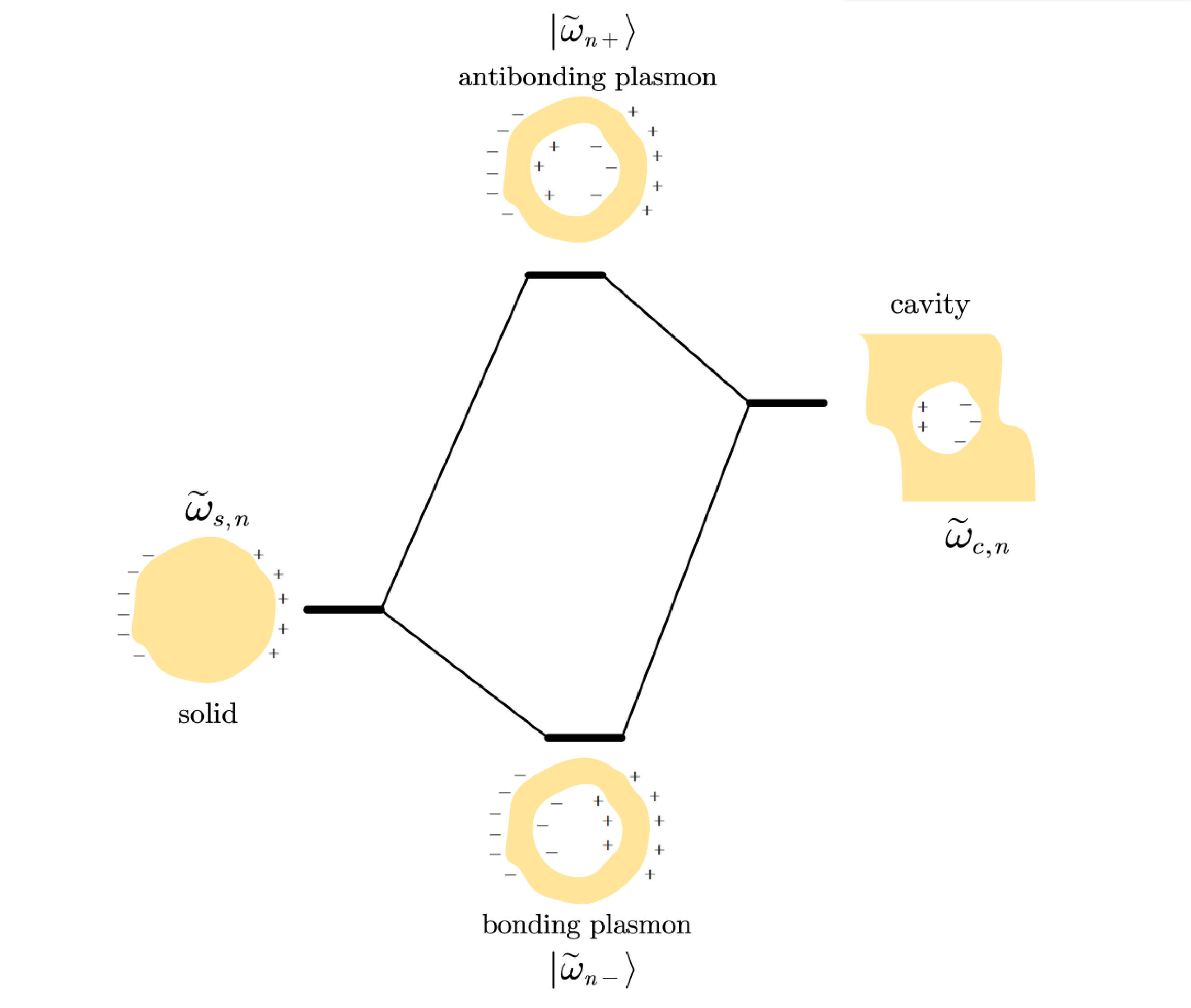}
    	\caption{Hybridization energy-level diagram induced by interaction between solid and cavity plasmons} 
    	\label{picture-hybridization}
    \end{figure}

\begin{rem}
    The mathematical results above provide a rigorous and insightful framework for interpreting the hybridization phenomenon in complex metallic nanostructures. This not only offers a theoretical characterization of the hybridization process but also enables quantitative predictions for the practical applications mentioned earlier. Specifically, the solid and cavity plasmon can be considered as free excitations of the system, which start interacting with each other when they are compounded into metal nanoshells, which leads to splitting, transferring and remixing of the plasmonic energy. This phenomenon we call the hybridization of solid and cavity plasmon modes. According to electronic structure theory, the hybridization mode is comparable to the process by which atomic orbitals generate molecular orbitals. In addition, the scales of the metal core and shell also affect the intensity of the solid and cavity plasmon interaction, which split the plasmon resonance into two new resonances: the high-energy antibonding plasmon $|\tilde{\omega}_{n+}\rangle$ and the low-energy bonding plasmon $|\tilde{\omega}_{n-}\rangle $. The $|\tilde{\omega}_{n+}\rangle$ mode corresponds to the antisymmetric coupling between the solid-cavity modes, and the $|\tilde{\omega}_{n-}\rangle $ mode corresponds to the symmetric coupling between the two modes. The energy-level diagram of the plasmon hybridization in the metal nanoshells induced by the interaction of the two plasmons, the solid and the cavity, is shown in Figure \ref{picture-hybridization}. Simplified again, one way to conceptualize the metal nanocore-shell structure's plasmon resonance is as a collection of plasmons that emerge from the much reduced geometry and interact to create a system \cite{prodan2004plasmon,prodan2003hybridization}.

    Identity (\ref{omega-p^2}) clearly clarifies that the solid sphere and cavity plasmon modes are generated by the free excitation of the system, and the hybridization between these two modes significantly influences the resonance frequency of the metallic nanoshell plasmon.
\end{rem}

\section{Auxiliary results}\label{sec:pre}
    In this section, we will introduce layer potentials, Neumann-Poincar\'{e} operator and its scale invariants related to electromagnetic scattering. Furthermore, we will introduce the NP-type operator and state its relevant properties. A systematic review and synthesis of this preparatory knowledge will help the reader to have a comprehensive grasp of this study.
\subsection{Layer potentials}
    We denote by $\varGamma \left( x,y \right)$ the fundamental solution for the Laplacian in the free space. In $\mathbb{R} ^2$, we have
    \begin{equation*}		
	    \varGamma \left( x,y \right) =\dfrac{1}{2\pi}\ln |x-y|.
    \end{equation*}

    Consider a bounded $\mathcal{C}^{2}$-domain $D$ in $\mathbb{R} ^2$. Suppose that $D$ contains the origin 0. The single-layer potential operator  $\mathcal{S} _D$ and double-layer potential operator  $\mathcal{D} _D$ associated with $D$ are given by
    \begin{align*}
    	\mathcal{S} _D\left[ \varphi \right] \left( x \right) & =\int_{\partial D}{\varGamma \left( x,y \right) \varphi \left( y \right) \d\sigma \left( y \right)},\quad  x\in \partial D, \\
    	\mathcal{D} _D\left[ \varphi \right] \left( x \right) & = \int_{\partial D}{\frac{\partial \varGamma \left( x,y \right)}{\partial \nu \left( y \right)}\varphi \left( y \right) \d \sigma \left( y \right)}, \quad x\in \partial D,
    \end{align*}
    where $\varphi \in L^2 \left( \partial D \right)$ is the density function, and $\nu \left( x \right)$ denotes the outward unit normal at $x \in \partial D$. when $x\in \mathbb{R} ^2\setminus \partial D$, $\mathcal{S} _D\left[ \varphi \right] \left( x \right)$ and $\mathcal{D} _D\left[ \varphi \right] \left( x \right)$ are referred to as the single-layer potential and the double-layer potential, respectively.

    The unit normal derivative operator $\mathcal{K} _{D}^{*}$ of the single-layer potential operator, defined by
    \begin{equation*}
    	\mathcal{K} _{D}^{*}\left[ \varphi \right] \left( x \right) =\int_{\partial D}{\frac{\partial \varGamma \left( x,y \right)}{\partial \nu \left( x \right)}\varphi \left( y \right) \d\sigma \left( y \right)}, \quad x\in \partial D,
    \end{equation*}
    and the $L^2$-adjoint operator $\mathcal{K} _{D}$ of $\mathcal{K} _{D}^{*}$ is denoted by
    \begin{equation*}
    	\mathcal{K} _D\left[ \varphi \right] \left( x \right) =\int_{\partial D}{\frac{\partial \varGamma \left( x,y \right)}{\partial \nu \left( y \right)}\varphi \left( y \right) \d\sigma \left( y \right)}, \quad x\in \partial D.
    \end{equation*}
    The boundary integral operators $\mathcal{K} _{D}^{*}$ and $\mathcal{K} _D$ are also known as Neumann-Poincar\'{e} operators, i.e., the NP operators.

    In the special case of the disk, the expressions of the operators $\mathcal{K} _{D}^{*}$ and $\mathcal{K} _D$ can be simplified \cite{ammari2013mathematical}.
    \begin{lem}
        Suppose that $D \subset \mathbb{R} ^2 $ is a disk with radius $r_0$, then,
        \begin{equation*}
        	\frac{\left< x-y,\nu _x \right>}{\left| x-y \right|^2}=\frac{1}{2r_0},\quad \forall x,y\in \partial D,x\ne y.
        \end{equation*}
        And therefore, for any $\varphi \in L^2\left( \partial D \right)$,
        \begin{equation*}
        	\mathcal{K} _{D}^{*}\left[ \varphi \right] \left( x \right) =\mathcal{K} _D\left[ \varphi \right] \left( x \right) =\frac{1}{4\pi r_0}\int_{\partial D}{\varphi \left( y \right) \d\sigma \left( y \right)},
        \end{equation*}
        for all $x\in \partial D$.
    \end{lem}

    For a function $u$ defined on $\mathbb{R} ^2\setminus \partial D$, we define
    \begin{equation*}
    	\left. u \right|_{\pm}\left( x \right) :=\lim_{t\rightarrow 0^+} u\left( x\pm t\nu \left( x \right) \right) , \quad x\in \partial D,
    \end{equation*}
    and
    \begin{equation*}
    	\left. \frac{\partial u}{\partial \nu} \right|_{\pm}\left( x \right) :=\lim_{t\rightarrow 0^+} \left< \nabla u\left( x\pm t\nu \left( x \right) \right) ,\nu \left( x \right) \right> ,\quad  x\in \partial D,
    	\vspace{1ex}
    \end{equation*}
    if the limits exist. Then we have \cite{colton2019inverse}
    \begin{equation}
    	\label{jump_relation1}
    	\left. \mathcal{S} _D\left[ \varphi \right] \right|_+\left( x \right) =\left. \mathcal{S} _D\left[ \varphi \right] \right|_-\left( x \right) ,\quad  x\in \partial D,
    \end{equation}
    \begin{equation}
    	\label{jump_relation2}
    	\left. \frac{\partial \mathcal{S} _D\left[ \varphi \right]}{\partial \nu} \right|_{\pm}\left( x \right) =\left( \pm \frac{1}{2}I+\mathcal{K} _{D}^{*} \right) \left[ \varphi \right] \left( x \right) ,\quad  x\in \partial D.
    	\vspace{1ex}
    \end{equation}
    From (\ref{jump_relation1}) and (\ref{jump_relation2}), the single-layer potential operator is continuous at the interface while its normal derivative operator is discontinuous. The equation (\ref{jump_relation2}) is known as the jump relation, which relates the normal derivatives of the single-layer potential operator to the NP operator $\mathcal{K} _{D}^{*}$.

    In order to study the effect of core and shell scales on plasmon resonance for metal nanoshells of general geometries, we introduce a scaling factor $\delta $ to describe them. Given an arbitrary geometrically structured nanoparticle $B$, its scale is denoted as $\left| B \right|=O\left( 1 \right)$, and the particle obtained after scaling by a factor $\delta $ in the same proportion is denoted as $D$. We use the scale factor $\delta $ to characterize the scale of $D$, and thus the original problem is transformed into a study of the scale factor $\delta $. The following lemma is important in the investigation of scale factors.
    \begin{lem}
    Assume bounded domains $D, B \subset \mathbb{R} ^2 $ satisfying $D=p+\delta B$, with $p$ indicating the position of the centre of the domain $D$. For any $x\in \partial D$, $\tilde{x}:=\frac{x-p}{\delta}\in \partial B$, and for each density function $f \in L_{0}^{2}\left( \partial D \right) $, we introduce an associated density function $\eta \left( f \right) \in L_{0}^{2}\left( \partial B \right) $ defined on $\partial B$, satisfying
    \begin{equation*}
    	    \eta \left( f \right) \left( \tilde{x} \right) =f\left( p+\delta \tilde{x} \right) ,
    \end{equation*}
    then
    \begin{align}
    	    \label{scale-invariant1}
    	    \mathcal{S} _D\left[ f \right] \left( x \right) & =\delta \mathcal{S} _D\left[ \eta \left( f \right) \right] \left( \tilde{x} \right),\\
    	    \label{scale-invariant2}
    	    \mathcal{K} _{D}^{*}\left[ f \right] \left( x \right) & =\mathcal{K} _{D}^{*}\left[ \eta \left( f \right) \right] \left( \tilde{x} \right) ,
    \end{align}
    where \textnormal{(\ref{scale-invariant2})} is called the scale invariance of the NP operator.
    \end{lem}
    \begin{proof}
    The proof takes (\ref{scale-invariant2}) as an example, and (\ref{scale-invariant1}) can be proved similarly.
    \begin{equation*}
        \mathcal{K} _{D}^{*}\left[ f \right] \left( x \right)
          =\frac{1}{2\pi}\int_{\partial D}{\frac{\left< x-y,\nu _x \right>}{\left| x-y \right|^2}f\left( y \right) \d\sigma \left( y \right)} \\
          =\frac{1}{2\pi}\int_{\partial B}{\frac{\left< \tilde{x}-\tilde{y},\nu _{\tilde{x}} \right>}{\left| \tilde{x}-\tilde{y} \right|^2}\eta \left( f \right) \left( \tilde{y} \right) \d\sigma \left( \tilde{y} \right)}
          =\mathcal{K} _{B}^{*}\left[ \eta \left( f \right) \right] \left( \tilde{x} \right).
    \end{equation*}
    \end{proof}

\subsection{Neumann-Poincar\'{e}-type operator}
    In order to precisely describe the plasmon resonance behavior of metal nanoshells, the bounded domain $D_1 \subset \mathbb{R} ^2 $ is denoted to represent the solid dielectric core with interface $ \partial D_1 $  and permittivity $\varepsilon _m $, the bounded domain $D_{2}\setminus \overline{D_{1}}$ denotes the noble metal shell, with interfaces $\partial D_{1}$ and $\partial D_{2}$, whose dielectric function $\varepsilon _s\left( \omega \right)$ can be described by the Drude's model.
    The domain $\mathbb{R} ^2\setminus \overline{D_{2}}$ denotes to represent the solid dielectric with the same permittivity $\varepsilon _m $, where $\left| D_{1} \right|=O\left( 1 \right)$, $\left| D_{2} \right|=O\left( 1 \right)$.
    The transmission conditions along the $\partial D_{1}$ and $\partial D_{2}$ lead to the following model of the electrostatic field
    \begin{equation}
    	\label{Math_Model}
        \left\{
    	\begin{aligned}
    		&\nabla \cdot \left( \varepsilon \nabla u\left( x \right) \right) =0, \quad &&x\in \mathbb{R} ^2, \vspace{1ex} \\
    		&\left. u\left( x \right) \right|_+=\left. u\left( x \right) \right|_-, \quad   &&x\in \partial D_{1}\cup \partial D_{2}, \vspace{1ex} \\
    		&\varepsilon _s\left( \omega \right) \left. \dfrac{\partial u}{\partial \nu}\left( x \right) \right|_+=\varepsilon _m\left. \dfrac{\partial u}{\partial \nu}\left( x \right) \right|_-, \quad  &&x\in \partial D_{1}, \vspace{1ex} \\
    		&\varepsilon _m\left. \dfrac{\partial u}{\partial \nu}\left( x \right) \right|_+=\varepsilon _s\left( \omega \right) \left. \dfrac{\partial u}{\partial \nu}\left( x \right) \right|_-, \quad &&x\in \partial D_{2}, \vspace{1ex} \\
    		&u\left( x \right) -H\left( x \right) =O\left( |x|^{-1} \right),  \quad &&|x|\rightarrow \infty,\\
    	\end{aligned}
        \right.
    \end{equation}
    where
    \begin{equation*}
    	\varepsilon =\varepsilon _m\chi \left( D_{1}\cup \mathbb{R} ^2\setminus \overline{D_{2}} \right) +\varepsilon _s\left( \omega \right) \chi \left( D_{2}\setminus \overline{D_{1}} \right).
    \end{equation*}
     The total field $u$ can be represented as
    \begin{equation*}
    	u\left( x \right) =H\left( x \right) +\mathcal{S} _{D_{1}}\left[ \varphi \right] \left( x \right) +\mathcal{S} _{D_{2}}\left[ \phi \right] \left( x \right) .
    \end{equation*}
    From the transmission conditions in (\ref{Math_Model}) and the jump relation (\ref{jump_relation2}), it follows that the potential pair $\left(\varphi , \phi\right)$ is the solution of the following boundary integral equations
    \begin{equation}
        \label{tbc1}
        \left\{
    	\begin{aligned}    		
    		&\left( z\left( \omega \right) I-\mathcal{K} _{D_{1}}^{*} \right) \left[ \varphi \right] \left( x \right) -\dfrac{\partial \mathcal{S} _{D_{2}}\left[ \phi \right]}{\partial \nu _{1}} \left( x \right) =\dfrac{\partial H}{\partial \nu _{1}}\left( x \right),  \quad && x\in \partial D_{1}, \vspace{2ex} \\
    		&\dfrac{\partial \mathcal{S} _{D_{1}}\left[ \varphi \right]}{\partial \nu _{2}} \left( x \right) +\left( z\left( \omega \right) I+\mathcal{K} _{D_{2}}^{*} \right) \left[ \phi \right] \left( x \right) =-\dfrac{\partial H}{\partial \nu _{2}}\left( x \right),  \quad && x\in \partial D_{2},\\
    	\end{aligned}
        \right.
    \end{equation}
    where $\left(\varphi , \phi\right)$ is defined on $L^{2}_{0}\left( \partial D_{1} \right) \times L^{2}_{0}\left( \partial D_{2} \right) $ and $\nu _{1}$, $\nu _{2}$ denote the outward unit normal vectors of $\partial D_{1}$, $\partial D_{2}$ respectively, and $z\left( \omega \right)$ is denoted by
    \begin{equation}
    	\label{z_omega}
    	z\left( \omega \right) =\frac{\varepsilon _m+\varepsilon _s\left( \omega \right)}{2\left( \varepsilon _m-\varepsilon _s\left( \omega \right) \right)}.
    \end{equation}
    The  boundary integral equations (\ref{tbc1}) can be rewritten as
    \begin{equation}
    	\label{martix1}
    	\left[ \begin{matrix}
    		z\left( \omega \right) I-\mathcal{K} _{D_{1}}^{*} &	-\dfrac{\partial \mathcal{S} _{D_{2}}}{\partial \nu _{1}} \vspace{2ex} \\
    		\dfrac{\partial \mathcal{S} _{D_{1}}}{\partial \nu _{2}} & z\left( \omega \right) I+\mathcal{K} _{D_{2}}^{*}\\
    	\end{matrix} \right] \left[ \begin{array}{c}
    		\varphi \vspace{2ex}\\
    		\phi\\
    	\end{array} \right] =\left[ \begin{array}{c}
    		\dfrac{\partial H}{\partial \nu _{1}}  \vspace{2ex} \\
    		-\dfrac{\partial H}{\partial \nu _{2}}\\
    	\end{array} \right].
    \end{equation}
    Let the Neumann-Poincar\'{e}-type operator $\mathbb{K} ^*$, i.e., NP-type operator be defined by
    \begin{equation*}
    	\mathbb{K} ^*=\left[ \begin{matrix}
    		-\mathcal{K} _{D_{1}}^{*}&		-\dfrac{\partial \mathcal{S} _{D_{2}}}{\partial \nu _{1}} \vspace{2ex} \\
    		\dfrac{\partial \mathcal{S} _{D_{1}}}{\partial \nu _{2}}&		\mathcal{K} _{D_{2}}^{*}\\
    	\end{matrix} \right] .
    \end{equation*}
    Obviously, the boundary integral equations (\ref{martix1}) have the following simpler form
    \begin{equation*}
    	\label{tbc2}
    	\left( z\left( \omega \right) \mathbb{I} +\mathbb{K} ^* \right) \left[\varPhi\right] =f,
    \end{equation*}
    where
    \begin{equation*}
    	\mathbb{I} =\left[ \begin{matrix}
    		I&	\quad	0 \vspace{2ex}\\
    		0&	\quad	I\\
    	\end{matrix} \right] ,\quad
    	\varPhi =\left[ \begin{array}{c}
    		\varphi \vspace{2ex} \\
    		\phi\\
    	\end{array} \right] ,\quad
    	f=\left[ \begin{array}{c}
    		\dfrac{\partial H}{\partial \nu _{1}}  \vspace{2ex} \\
    		-\dfrac{\partial H}{\partial \nu _{2}}\\
    	\end{array} \right].
    \end{equation*}

    The NP-type operator has the following properties \cite{ammari2013spectral,kang2015layer}:
    \begin{enumerate}
    	\item[(1)] The spectrum of the $\mathbb{K} ^*$ lies in the interval $\left[ -\frac{1}{2},\frac{1}{2} \right]$ .
    	\item[(2)] The $L^2-$ adjoint operator of $\mathbb{K} ^*$, $\mathbb{K} $ is denoted as
    	\begin{equation*}
    		\mathbb{K} =\left[ \begin{matrix}
    			-\mathcal{K} _{D_{1}}&		\mathcal{D} _{D_{2}} \vspace{2ex} \\
    			-\mathcal{D} _{D_{1}}&		\mathcal{K} _{D_{2}}\\
    		\end{matrix} \right],
    	\end{equation*}
    	where $\mathcal{D}_{D_1}$, $\mathcal{D}_{D_2}$ are the double-layer potentials associated with $D_1$ and $D_2$, respectively.
    	\item[(3)] The operator $\mathbb{S}$ is defined as
    	\begin{equation*}
    		\mathbb{S} =\left[ \begin{matrix}
    			\mathcal{S} _{D_{1}}&		\mathcal{S} _{D_{2}} \vspace{2ex} \\
    			\mathcal{S} _{D_{1}}&		\mathcal{S} _{D_{2}}\\
    		\end{matrix} \right],
    	\end{equation*}
    	is self-adjoint, and $-\mathbb{S} \geqq  0$ on $L_{0}^{2}\left( \partial D_{1} \right) \times L_{0}^{2}\left( \partial D_{2} \right) $.
    	\item[(4)] Calder\'{o}n identity holds
    	\begin{equation*}
    		\mathbb{S} \mathbb{K} ^*=\mathbb{K} \mathbb{S}.
    	\end{equation*}
    	i.e., $\mathbb{S} \mathbb{K} ^*$ is self-adjoint.
    \end{enumerate}

    These properties comprise an inner product mapping defined on $L_{0}^{2}\left( \partial D_{1} \right) \times L_{0}^{2}\left( \partial D_{2} \right)$,
    \begin{equation*}
    	\left( \Phi ,\Psi \right) \mapsto \left< \Phi ,-\mathbb{S} \left[ \Psi \right] \right> =-\left< \varphi _1,-\mathcal{S} _{D_{\delta _1}}\left[ \psi _1 \right] \right> -\left< \varphi _2,-\mathcal{S} _{D_{\delta _2}}\left[ \psi _2 \right] \right>.
    \end{equation*}
    Let $\mathcal{H} ^2=L_{0}^{2}\left( \partial D_{1} \right) \times L_{0}^{2}\left( \partial D_{2} \right)$ be a Hilbert space with an inner product as above, i.e.,
    \begin{equation}
    	\label{H^2_product}
    	\left< \Phi ,\Psi \right> _{\mathcal{H} ^2}:=\left< \Phi ,-\mathbb{S} \left[ \Psi \right] \right> .
    \end{equation}
    From the Calder\'{o}n identity, the NP-type operator $\mathbb{K}^*$ is self-adjoint on the space $\mathcal{H}^2$ .Let $\lambda _1,\lambda _2,\cdots $ be the non-zero eigenvalues of $\mathbb{K}^*$ ($\left|\lambda_1 \right|\ge \left| \lambda_2 \right|  \ge \cdots$, counting multiplcities), and $\Psi_n$ be the corresponding normalized eigenfunction, i.e. $\left\| \Psi_n \right\| _{\mathcal{H} ^2}=1$. Then, we have
    \begin{equation*}
    	\left< \Psi _i,\Psi _j \right> _{\mathcal{H} ^2}=\delta _{ij},
    \end{equation*}
    where $\delta _{ij}$ denotes the Kronecker delta symbol, and $\mathbb{K} ^*$ has the following spectral decomposition
    \begin{equation*}
    	\mathbb{K} ^*\left[ \Phi \right] =\sum_{n=1}^{\infty}{\lambda _n\left< \Phi ,\Psi _n \right> _{\mathcal{H} ^2}\Psi _n}, \quad
    \Phi \in \mathcal{H} ^2.
    \end{equation*}
    Since $\mathbb{K}^*$ is a Hilbert-Schmidt operator, the solution $\varPhi =\left( \varphi ,\phi \right)^T$ of the (\ref{martix1}) can be represented as
    \begin{equation}
    	\label{Phi}
    	\varPhi =\sum_{n=1}^{\infty}{\dfrac{\left< f,\Psi _n \right> _{\mathcal{H} ^2}}{\lambda _n+z\left( \omega \right)}\Psi _n} .
    \end{equation}

\subsection{Plasmon resonance}
    For the mathematical definition of plasmon resonance is mentioned in several references \cite{ammari2013spectral,ammari2017mathematical,fang2023plasmon}, by spectral decomposition of the $\mathbb{K} ^*$, we obtain the solution of the boundary integral equations (\ref{martix1}) in the form (\ref{Phi}), which is non-zero due to the non-zero value of $\left< f,\Psi _n \right> _{\mathcal{H} ^2} $, when $- z\left( \omega \right)$ converges infinitely to the eigenvalue of the $\mathbb{K} ^*$ , the interface density function $\varPhi$ will undergo a rapid bursting, and the scattering field $u-H$ will exhibit resonant behavior. From this, we define the following plasmon resonance condition.
    \begin{defn}
    Let $\lambda _1,\lambda _2,\cdots $ be the non-zero eigenvalues of the NP-type operator $\mathbb{K} ^*$, satisfying $\left| \lambda _1 \right|\ge \left| \lambda _2 \right| \ge \cdots$, counting multiplcities. Then plasmon resonance occurs if the $z\left( \omega \right)$ in (\ref{z_omega}) is fulfilled:
    \begin{equation}
    	    \label{resonance_condition}
    	    \left|  z\left( \omega \right) + \lambda _n \right| \ll 1.
    \end{equation}
    The inequality (\ref{resonance_condition}) is called the plasmon resonance condition, and the corresponding frequency $\omega$ is called the plasmonic resonant frequency.
    \end{defn}
    \begin{rem}
    Each eigenvalue of the NP-type operator $\mathbb{K} ^*$ corresponds to a resonance mode, so the plasmonic resonant frequency is not unique.
    \end{rem}
    \begin{defn}
    Assume that the non-zero eigenvalues of the NP-type operator are of the form
    \begin{equation*}
    \lambda _n^{(0)}+\epsilon \lambda _n^{(1)}+\mathcal {O}(\epsilon ^2),\quad n=1,2,\cdots.
    \end{equation*}
    We say that $\omega$ is a first-order corrected plasmon resonance frequency if
    \begin{equation}
    \label{first-order-resonance-condition}
    \left|  z\left( \omega \right) +\lambda _n^{(0)}+\epsilon \lambda _n^{(1)} \right| \ll 1.
    \end{equation}
    The inequality (\ref{first-order-resonance-condition}) is called the first-order corrected plasmon resonance condition.
    \end{defn}
\section{Perturbations of concentric disk}\label{sec:pertur}
\subsection{Rescaling and asymptotic analysis}
    As described in section \ref{sec:math-setting}, the total field $u$ satisfies model (\ref{perturbations_math_model}) due to electromagnetic scattering. Along the interfaces $\partial D_{\epsilon _1,\delta _1}$ and $\partial D_{\epsilon _2,\delta _2}$, the transmission conditions that are fulfilled
    \begin{align*}
    	\begin{cases}
    		\varepsilon _s\left( \omega \right) \left. \dfrac{\partial u}{\partial \nu}\left( x \right) \right|_+=\varepsilon _m\left. \dfrac{\partial u}{\partial \nu}\left( x \right) \right|_-, \quad &x\in \partial D_{\epsilon _1,\delta _1}, \vspace{1ex} \\
    		\varepsilon _m\left. \dfrac{\partial u}{\partial \nu}\left( x \right) \right|_+=\varepsilon _s\left( \omega \right) \left. \dfrac{\partial u}{\partial \nu}\left( x \right) \right|_-,  \quad &x\in \partial D_{\epsilon _2,\delta _2}. \\
    	\end{cases}
    \end{align*}
    Then the potential pair $\left( \bar{\varphi},\bar{\phi} \right)$ is a solution to the following boundary integral equations
    \begin{equation}
        \label{martix5.1}
        \left\{
    	\begin{aligned}    		
    		&\left( z\left( \omega \right) I-\mathcal{K} _{D_{\epsilon _1,\delta _1}}^{*} \right) \left[ \bar{\varphi} \right] \left( x \right) -\dfrac{\partial \mathcal{S} _{D_{\epsilon _2,\delta _2}}\left[ \bar{\phi} \right]}{\partial \bar{\nu}_1}  \left( x \right) =\dfrac{\partial H}{\partial \bar{\nu}_1}\left( x \right),  \quad && x\in \partial D_{\epsilon _1,\delta _1}, \vspace{1ex} \\
    		&\dfrac{\partial \mathcal{S} _{D_{\epsilon _1,\delta _1}}\left[ \bar{\varphi} \right]}{\partial \bar{\nu}_2} \left( x \right) +\left( z\left( \omega \right) I+\mathcal{K} _{D_{\epsilon _2,\delta _2}}^{*} \right) \left[ \bar{\phi} \right] \left( x \right) =-\dfrac{\partial H}{\partial \bar{\nu}_2}\left( x \right),  \quad && x\in \partial D_{\epsilon _2,\delta _2},\\
    	\end{aligned}
        \right.
    \end{equation}
    where the formula for $z\left( \omega \right)$ is the same as (\ref{z_omega}). Then (\ref{martix5.1}) can be rewritten to
    \begin{equation*}
    	\left( z\left( \omega \right) \mathbb{I} +\left[ \begin{matrix}
    		-\mathcal{K} _{D_{\epsilon _1,\delta _1}}^{*} & -\dfrac{\partial \mathcal{S} _{D_{\epsilon _2,\delta _2}}}{\partial \bar{\nu}_1} \vspace{2ex} \\
    		\dfrac{\partial \mathcal{S} _{D_{\epsilon _1,\delta _1}}}{\partial \bar{\nu}_2} & \mathcal{K} _{D_{\epsilon _2,\delta _2}}^{*}\\
    	\end{matrix} \right] \right) \left[ \begin{array}{c}
    		\bar{\varphi} \vspace{2ex} \\
    		\bar{\phi}\\
    	\end{array} \right] =\left[ \begin{array}{c}
    		\dfrac{\partial H}{\partial \bar{\nu}_1} \vspace{2ex} \\
    		-\dfrac{\partial H}{\partial \bar{\nu}_2}\\
    	\end{array} \right] .
    \end{equation*}
    Now, we will define the NP-type operator
    \begin{equation*}
    	\mathbb{K}_{\delta_1,\delta_2 ,\epsilon_1, \epsilon_2} ^*= \left[ \begin{matrix}
    		-\mathcal{K} _{D_{\epsilon _1,\delta _1}}^{*} & -\dfrac{\partial \mathcal{S} _{D_{\epsilon _2,\delta _2}}}{\partial \bar{\nu}_1} \vspace{2ex} \\
    		\dfrac{\partial \mathcal{S} _{D_{\epsilon _1,\delta _1}}}{\partial \bar{\nu}_2} & \mathcal{K} _{D_{\epsilon _2,\delta _2}}^{*}\\
    	\end{matrix} \right].
    \end{equation*}

    In order to explore the crucial influence of the scale factors $\delta _1$ and $\delta _2$ on the resonance frequency, we perform a scale transformation of (\ref{martix5.1}) and use the scale invariant property of the NP operator. For this purpose, we first introduce two associated functions. For the density functions $\bar{\varphi}\left( \bar{x} \right)$ and $\bar{\phi}\left( \bar{x} \right)$ defined on $\partial D_{\epsilon _1,\delta _1}$ and $\partial D_{\epsilon _2,\delta _2}$, respectively, introduce the associated density functions $\tilde{\eta}_1\left( \varphi \right) \left( \tilde{x} \right)$ and $\tilde{\eta}_2\left( \phi \right) \left( \tilde{x} \right)$ defined on $\partial D_{\epsilon _1}$ and $\partial D_{\epsilon _2}$, which satisfy
    \begin{align*}
    	\tilde{\eta}_1\left( \varphi \right) \left( \tilde{x} \right) &  =\bar{\varphi}\left( z+\delta _1\tilde{x} \right) , \\
    	\tilde{\eta}_2\left( \phi \right) \left( \tilde{x} \right) &  =\bar{\phi}\left( z+\delta _2\tilde{x} \right) .
    \end{align*}

    Based on the scale invariance property of the NP operator (\ref{scale-invariant2}), we obtain the following results
    \begin{align*}
    	& \mathcal{K} _{D_{\epsilon _1,\delta _1}}^{*}\left[ \bar{\varphi} \right] \left( \bar{x} \right)   =\mathcal{K} _{D_{\epsilon _1}}^{*}\left[ \tilde{\eta}_1\left( \varphi \right) \right] \left( \tilde{x} \right) , \ \  \bar{x}\in \partial D_{\epsilon _1,\delta _1},\tilde{x}\in \partial D_{\epsilon _1}, \\
        & \mathcal{K} _{D_{\epsilon _2,\delta _2}}^{*}\left[ \bar{\phi} \right] \left( \bar{x} \right)   =\mathcal{K} _{D_{\epsilon _2}}^{*}\left[ \tilde{\eta}_2\left( \phi \right) \right] \left( \tilde{x} \right) , \ \  \bar{x}\in \partial D_{\epsilon _2,\delta _2},\tilde{x}\in \partial D_{\epsilon _2}.
    \end{align*}
    Considering the center of the $D_{\epsilon _1,\delta _1}$ and $D_{\epsilon _2,\delta _2}$ at the origin, the transmission conditions (\ref{martix5.1}) are equivalent to
    \begin{equation}
    	\label{tbc5.2}
        \left\{
    	\begin{aligned}
    		&\left( z\left( \omega \right) I-\mathcal{K} _{D_{\epsilon _1}}^{*} \right) \left[ \tilde{\eta}_1\left( \varphi \right) \right] \left( \tilde{x} \right) -\dfrac{\delta _2}{\delta _1}\dfrac{\partial \mathcal{S} _{D_{\epsilon _2}}\left[ \tilde{\eta}_2\left( \phi \right) \right]}{\partial \tilde{\nu}_1} \left(\dfrac{\delta _1}{\delta _2}\tilde{x} \right) =\dfrac{1}{\delta _1}\dfrac{\partial H\left( \delta _1\tilde{x} \right)}{\partial \tilde{\nu}_1}, \quad  && \tilde{x}\in \partial D_{\epsilon _1},
    		\vspace{2ex} \\
    		&\dfrac{\delta _1}{\delta _2}\dfrac{\partial \mathcal{S} _{D_{\epsilon _1}}\left[ \tilde{\eta}_1\left( \varphi \right) \right]}{\partial \tilde{\nu}_2} \left( \dfrac{\delta _2}{\delta _1} \tilde{x} \right) +\left( z\left( \omega \right) I+\mathcal{K} _{D_{\epsilon _2}}^{*} \right) \left[ \tilde{\eta}_2\left( \phi \right) \right] \left( \tilde{x} \right) =-\dfrac{1}{\delta _2}\dfrac{\partial H\left( \delta _2\tilde{x} \right)}{\partial \tilde{\nu}_2},  \quad &&\tilde{x}\in \partial D_{\epsilon _2}.\\
    	\end{aligned}
        \right.
    \end{equation}

    For $i=1,2$, $0< \beta <1$, let $X_i\left( t \right) :\left[ 0, 2\pi \right] \rightarrow \mathbb{R} ^2$ be of class $C^{2, \beta}$, satisfying $\left| X_{i}^{'}\left( t \right) \right|=1$ for all $t\in \left[ 0,2\pi \right] $. We consider that $D_1$ and $D_2$ are both bounded $C^{2, \beta}$-domains in $\mathbb{R} ^2$, with interface $\partial D_i$ parametrized by $X_i\left( t \right)$
    \begin{equation*}
    	\partial D_i=\left\{ x~|~x=X_i\left( t \right) ,t\in \left[ 0,2\pi \right] \right\} ,
    \end{equation*}
    the outward unit normal vectors $\nu _i$ on $\partial D_i$ is given by
    \begin{equation*}
    	\nu _i\left( x \right) =R_{-\frac{\pi}{2}}X_{i}^{'}\left( t \right) ,
    \end{equation*}
    where $R_{-\frac{\pi}{2}}$ denotes the rotation $-\pi/2$. Furthermore, the tangential vector at $x$ is defined as $T_i\left( x \right)=X_{i}^{'}\left( t \right) $, satisfying $X_{i}^{'}\left( t \right) \bot X_{i}^{''}\left( t \right) $. The curvature $\tau _i \left( t \right) $ is defined by
    \begin{equation*}
    	X_{i}^{''}\left( t \right) =\tau _i\left( t \right) \nu _i\left( x \right) .
    \end{equation*}
    We use $h_i\left( t \right) $ for $h_i\left( X_i\left( t \right) \right) $, and $h_{i}^{'}\left( t \right)$ for the   tangential derivatives of $h_i\left( x \right)$. The $\epsilon _i$-normal perturbation of $D_i$ yields the domain $D_{\epsilon _i}$, whose interface $\partial D_{\epsilon _i}$ can be defined as
    \begin{equation*}
    	\partial D_{\epsilon _i}=\left\{ \tilde{x}~|~\tilde{x}= x +\epsilon _ih_i\left( x \right) \nu _i\left( x \right) =X_i\left( t \right) +\epsilon _ih_i\left( t \right) R_{-\frac{\pi}{2}}X_{i}^{'}\left( t \right), x \in \partial D_i \right\} ,
    \end{equation*}
    where the function $h_i\in C^{2,\beta}\left( \partial D_i \right) $. Zribi has provided the proofs for the subsequent lemmas while conducting sensitivity analysis utilizing layer potential techniques \cite{ammari2010conductivity}.
    \begin{lem}
    \label{lemma4.1}
    For $\tilde{\nu}_i\left( \tilde{x} \right)$, the outward unit normal vector on $\partial D_{\epsilon _i}$ at $\tilde{x}$, we have
    \begin{equation*}
    	\tilde{\nu}_i\left( \tilde{x} \right) =\frac{\left[ 1-\epsilon _ih_i\left( t \right) \tau _i\left( x \right) \right] \nu _i\left( x \right) -\epsilon _ih_{i}^{'}\left( t \right) T_i\left( x \right)}{\sqrt{\left[ \epsilon _ih_{i}^{'}\left( t \right) \right] ^2+\left[ 1-\epsilon _ih_i\left( t \right) \tau _i\left( x \right) \right] ^2}},
    \end{equation*}
    and hence $\tilde{\nu}_i\left( \tilde{x} \right)$ can be  expanded uniformly as
    \begin{equation*}
    	\tilde{\nu}_i\left( \tilde{x} \right) =\sum_{n=0}^{\infty}{\epsilon _{i}^{n}\nu _{i}^{\left( n \right)}\left( x \right)}, \quad x\in \partial D_i,
    \end{equation*}
    where the vector-valued function $\nu _{i}^{\left( n \right)}$ is bounded. In particular, the first two terms are given by
    \begin{equation*}
    	\nu _{i}^{\left( 0 \right)}\left( x \right) =\nu _i\left( x \right) , \quad \nu _{i}^{\left( 1 \right)}\left( x \right) =-h_{i}^{'}\left( t \right) T_i\left( x \right) ,
    \end{equation*}
    then
    \begin{equation}
    	\label{nv_i_expansion}
    	\tilde{\nu}_i\left( \tilde{x} \right) =\nu _i\left( x \right) -\epsilon _ih_{i}^{'}\left( t \right) T_i\left( x \right) +O\left( \epsilon _{i}^{2} \right) .
    \end{equation}
    \end{lem}
    \begin{lem}
    For $\d\sigma _{\epsilon _i}\left( \tilde{y} \right) $, the length element on $\partial D_{\epsilon _i}$ at $\tilde{x}$, we have
    \begin{equation*}
    	\d\sigma _{\epsilon _i}\left( \tilde{y} \right) = \sqrt{\left[ 1-\epsilon _i\tau _i\left( s \right) h_i\left( s \right) \right] ^2+\left[ \epsilon _ih_{i}^{'}\left( s \right) \right] ^2}\d s,
    \end{equation*}
    and hence $\d\sigma _{\epsilon _i}\left( \tilde{y} \right) $ can be  expanded uniformly as
    \begin{equation*}
    	\d\sigma _{\epsilon _i}\left( \tilde{y} \right) =\sum_{n=0}^{\infty}{\epsilon _{i}^{n}\sigma _{i}^{\left( n \right)}\left( y \right)}\d\sigma \left( y \right),
    \end{equation*}
    where $\sigma _{i}^{\left( n \right)}$ is bounded. In particular, the first two terms are given by
    \begin{equation*}
    	\sigma _{i}^{\left( 0 \right)}\left( y \right) =1, \quad \sigma _{i}^{\left( 1 \right)}\left( y \right) =-\tau _i\left( y \right) h_i\left( y \right) ,
    \end{equation*}
    then
    \begin{equation*}
    	\d\sigma _{\epsilon _i}\left( \tilde{y} \right) =\left[ 1-\epsilon _i\tau _i\left( y \right) h_i\left( y \right) +O\left( \epsilon _{i}^{2} \right) \right] \d\sigma \left( y \right) .
    \end{equation*}
    \end{lem}

    Let $x=X_i\left( t \right)$, $\tilde{x}=x+\epsilon _ih_i\left( t \right) \nu _i\left( x \right)$, $y=X_i\left( s \right)$, $\tilde{y}=y+\epsilon _ih_i\left( s \right) \nu _i\left( y \right)$, we have
    \begin{equation*}
    	\tilde{x}-\tilde{y}=x-y+\epsilon _i\left( h_i\left( t \right) \nu _i\left( x \right) -h_i\left( s \right) \nu _i\left( y \right) \right) ,
    \end{equation*}
    furthermore
    \begin{equation}
    	\label{1/|tilde{x-y}|^2_expansion}
    	\frac{1}{\left| \tilde{x}-\tilde{y} \right|^2}=\frac{1}{\left| x-y \right|^2}\frac{1}{1+2\epsilon _iF_i\left( x,y \right) +\epsilon _{i}^{2}G_i\left( x,y \right)},
    \end{equation}
    where $F_i\left( x,y \right)$ and $G_i\left( x,y \right)$ are denoted by
    \begin{align*}
    	F_i\left( x,y \right) & =\frac{\left< x-y,h_i\left( t \right) \nu _i\left( x \right) -h_i\left( s \right) \nu _i\left( y \right) \right>}{\left| x-y \right|^2}, \\
    	G_i\left( x,y \right) & =\frac{\left| h_i\left( t \right) \nu _i\left( x \right) -h_i\left( s \right) \nu _i\left( y \right) \right|^2}{\left| x-y \right|^2}.
    \end{align*}
    \begin{lem}
    From \textnormal{(\ref{nv_i_expansion})} and \textnormal{(\ref{1/|tilde{x-y}|^2_expansion})}, we obtain the following uniform expansion
    \begin{equation*}
    	\frac{\left< \tilde{x}-\tilde{y},\tilde{\nu}_i\left( \tilde{x} \right) \right>}{\left| \tilde{x}-\tilde{y} \right|^2}\d\sigma _{\epsilon _i}\left( \tilde{y} \right) =\sum_{n=0}^{\infty}{\epsilon _{i}^{n}}K_{i}^{(n)}\left( x,y \right) \d\sigma \left( y \right) ,
    \end{equation*}
    where
    \begin{align*}
    	& K_{i}^{(0)}\left( x,y \right) =
    	\dfrac{\left< x-y,\nu _i\left( x \right) \right>}{\left| x-y \right|^2} , \\[1mm]
        & K_{i}^{(1)}\left( x,y \right) =
        K_{i}^{(0)}\left( x,y \right) \cdot P_i\left( x,y \right) +Q_i\left( x,y \right) , \\[1mm]
    	& P_i\left( x,y \right) =
    	-2F_i\left( x,y \right) +h_i\left( x \right) \nu _i\left( x \right) -h_i\left( y \right) \nu _i\left( y \right) ,  \\[1mm]
    	& Q_i\left( x,y \right) =
    	\dfrac{\left< h_i\left( t \right) \nu _i\left( x \right) -h_i\left( s \right) \nu _i\left( y \right) ,\nu _i\left( x \right) \right> }{\left| x-y \right|^2}  -\dfrac{\left< x-y,\tau _i\left( x \right) h_i\left( t \right) \nu _i\left( x \right) +h_{i}^{'}\left( t \right) T_i\left( x \right) \right>}{\left| x-y \right|^2}.
    \end{align*}
    \end{lem}

    Now, we introduce the sequence of integral operators $\left\{\mathcal{K} _{D_i}^{\left( n \right)} \right\}  _{n\in \mathbb{N}}$, for any $\phi \in L^2\left( \partial D_i \right)$, we define
    \begin{equation*}
    	\mathcal{K} _{D_i}^{\left( n \right)}\phi \left( x \right) =\frac{1}{2\pi}\int_{\partial D_i}{K_{i}^{(n)}\left( x,y \right) \phi \left( y \right) \d\sigma \left( y \right)}, \quad \forall n\ge 0.
    \end{equation*}
    Obviously, $ \mathcal{K} _{D_i}^{\left( 0 \right)}=\mathcal{K} _{D_i}^{*}$, and $\mathcal{K} _{D_i}^{\left( n \right)}$ is bounded on $L^2\left( \partial D_i \right)$. Meanwhile, the diffeomorphism $\Psi _{\epsilon _i}$ from $\partial D_i$ to $\partial D_{\epsilon _i}$ to be defined by
    \begin{equation*}
    	\Psi _{\epsilon _i}\left( x \right) =x+\epsilon _ih_i\left( t \right) \nu _i\left( x \right) ,
    \end{equation*}
    where $x=X_i\left( t \right)$.
    \begin{lem}
    \label{lemma4.4}
    Let $N\in \mathbb{N}$, there exist constants $C$ depending only on $N$, $C^2$-norm of $\partial D_i$, and $C^1$-norm of $h_i$, such that for any $\tilde{\phi}\in L^2\left( \partial D_{\epsilon _i} \right)$, we obtain
    \begin{equation*}
    	\left\| \left( \mathcal{K} _{D_{\epsilon _i}}^{*}\tilde{\phi} \right) \circ \Psi _{\epsilon _i}-\mathcal{K} _{D_i}^{*}\phi -\sum_{n=0}^{N-1}{\epsilon _{i}^{n+1}}\mathcal{K} _{D_i}^{\left( n+1 \right)}\phi \right\| _{L^2\left( \partial D_i \right)}\le C\epsilon _{i}^{N+1}\left\| \phi \right\| _{L^2\left( \partial D_i \right)},
    \end{equation*}
    where $\phi :=\tilde{\phi}\circ \Psi _{\epsilon _i}$, and $\mathcal{K} _{D_i}^{\left( n+1 \right)}$ is the compact operator on $L^2\left( \partial D_i \right)$. In particular, if $\phi \in C^{1, \beta}\left( \partial D_i \right)$, then we have
    \begin{align*}
    	\mathcal{K} _{D_i}^{\left( 1 \right)}=& -\mathcal{K} _{D_i}^{*}\left[ \tau _ih_i\phi \right] \left( x \right) +h_i\left( t \right) \left< D^2 \mathcal{S} _{D_i}\left[ \phi \right] \left( x \right)\nu _i\left( x \right) ,\nu _i\left( x \right) \right> \nonumber \\
    	&+\frac{\partial  \mathcal{D} _{D_i}\left[ h_i\phi \right] }{\partial \nu _i}\left( x \right) -h_{i}^{'}\left( t \right) \frac{\partial \mathcal{S} _{D_i}\left[\phi \right]}{\partial T_i}\left( x \right) .
    \end{align*}
    \end{lem}

    To establish the connection between $\partial D_i$ and $\partial D_{\epsilon _i}$, we again define the diffeomorphism $\Psi _{\epsilon _i}:\partial D_i\rightarrow \partial D_{\epsilon _i}$, where $\Psi _{\epsilon _i}\left( x \right) =x+\epsilon _ih_i\left( t \right) \nu _i\left( x \right) $, $x=X_i\left( t \right)$, and introduce
    \begin{align*}
    	\eta _1\left( \varphi \right) & =\tilde{\eta}_1\left( \varphi \right) \circ \Psi _{\epsilon _1}, \\
    	\eta _2\left( \phi \right) & =\tilde{\eta}_2\left( \phi \right) \circ \Psi _{\epsilon _2}.
    \end{align*}
    Considering that the plasmonic resonant frequency of metallic nanoshell is influenced by small factors $\epsilon _1$ and $\epsilon _2$, we will conduct the following asymptotic analysis.

    \begin{thm}
    \label{theorem4.1}
    Let $\eta _1\left( \varphi \right) \in C^{1,\beta}\left( \partial D_1 \right)$, and let $\eta _2\left( \phi \right) \in C^{1,\beta}\left( \partial D_2 \right) $. Then, for boundary integral operators $\mathcal{K} _{D_{\epsilon _1}}^{*}\left[ \tilde{\eta}_1\left( \varphi \right) \right] \left( \tilde{x} \right) $, $\tilde{x}\in \partial D_{\epsilon _1}$ and $\mathcal{K} _{D_{\epsilon _2}}^{*}\left[ \tilde{\eta}_2\left( \phi \right) \right] \left( \tilde{x} \right)$, $\tilde{x}\in \partial D_{\epsilon _2}$, we have the following asymptotic expansions hold:
    \begin{align*}
    	\mathcal{K} _{D_{\epsilon _1}}^{*}\left[ \tilde{\eta}_1\left( \varphi \right) \right] \left( \tilde{x} \right) & =\mathcal{K} _{D_{\epsilon _1}}^{*}\left[ \eta _1\left( \varphi \right) \right] \left( x \right) \nonumber \\
    	& =\mathcal{K} _{D_1}^{*}\left[ \eta _1\left( \varphi \right) \right] \left( x \right) +\epsilon _1\mathcal{K} _{D_1}^{\left( 1 \right)}\left[ \eta _1\left( \varphi \right) \right] \left( x \right) +\mathcal {O}\left( \epsilon _{1}^{2} \right),\\
        \mathcal{K} _{D_{\epsilon _2}}^{*}\left[ \tilde{\eta}_2\left( \phi \right) \right] \left( \tilde{x} \right) & =\mathcal{K} _{D_{\epsilon _1}}^{*}\left[ \eta _2\left( \phi \right) \right] \left( x \right) \nonumber \\
    	& =\mathcal{K} _{D_1}^{*}\left[ \eta _2\left( \phi \right) \right] \left( x \right) +\epsilon _2\mathcal{K} _{D_2}^{\left( 1 \right)}\left[ \eta _2\left( \phi \right) \right] \left( x \right) +\mathcal {O}\left( \epsilon _{2}^{2} \right).
    \end{align*}
    where
    \begin{align*}
    	\mathcal{K} _{D_1}^{\left( 1 \right)}\left[ \eta _1\left( \varphi \right) \right] \left( x \right) = & -\mathcal{K} _{D_1}^{*}\left[ \tau _1h_1\eta _1\left( \varphi \right) \right] \left( x \right) +h_1\left( t \right) \left< D^2 \mathcal{S} _{D_1}\left[ \eta _1\left( \varphi \right) \right] \left( x \right)\nu _1\left( x \right) ,\nu _1\left( x \right) \right> \nonumber \\
    	& +\frac{\partial  \mathcal{D} _{D_1}\left[ h_1 \eta _1\left( \varphi \right) \right] }{\partial \nu _1}\left( x \right) -h_{1}^{'}\left( t \right) \frac{\partial \mathcal{S} _{D_1}\left[ \eta _1\left( \varphi \right) \right]}{\partial T_1}\left( x \right) , \quad x\in \partial D_1,\\
        \mathcal{K} _{D_2}^{\left( 1 \right)}\left[ \eta _2\left( \phi \right) \right] \left( x \right) = &-\mathcal{K} _{D_2}^{*}\left[ \tau _2h_2\eta _2\left( \phi \right) \right] \left( x \right) +h_2\left( t \right) \left< D^2 \mathcal{S} _{D_2}\left[ \eta _2\left( \phi \right) \right] \left( x \right)\nu _2\left( x \right) ,\nu _2\left( x \right) \right> \nonumber \\
    	& +\frac{\partial  \mathcal{D} _{D_2}\left[ h_2 \eta _2\left( \phi \right)  \right] }{\partial \nu _2}\left( x \right) -h_{2}^{'}\left( t \right) \frac{\partial  \mathcal{S} _{D_2}\left[ \eta _2\left( \phi \right) \right] }{\partial T_2}\left( x \right) , \quad x\in \partial D_2.
    \end{align*}
    \end{thm}
    \begin{proof}
        This result is a consequence of Lemma \ref{lemma4.4}.
    \end{proof}

    \begin{thm}
    \label{theorem4.2}
    Let $\eta _1\left( \varphi \right) \in C^{1,\beta}\left( \partial D_1 \right)$, and let $\eta _2\left( \phi \right) \in C^{1,\beta}\left( \partial D_2 \right) $. Then, for $\dfrac{\partial \mathcal{S} _{D_{\epsilon _2}}\left[ \tilde{\eta}_2\left( \phi \right) \right]}{\partial \tilde{\nu}_1} \left( \dfrac{\delta _1}{\delta _2}\tilde{x} \right)$, $\tilde{x}\in \partial D_{\epsilon _1}$, and $\dfrac{\partial \mathcal{S} _{D_{\epsilon _1}}\left[ \tilde{\eta}_1\left( \varphi \right) \right]}{\partial \tilde{\nu}_2} \left( \dfrac{\delta _2}{\delta _1}\tilde{x} \right) $, $\tilde{x}\in \partial D_{\epsilon _2}$, we have the following dual asymptotic expansions with respect to $\epsilon _1$ and $\epsilon _2$ hold:
    \begin{align*}
    	& \dfrac{\partial \mathcal{S} _{D_{\epsilon _2}}\left[ \tilde{\eta}_2\left( \phi \right) \right]}{\partial \tilde{\nu}_1} \left( \dfrac{\delta _1}{\delta _2}\tilde{x} \right) \nonumber \\
    	= & \frac{\partial \mathcal{S} _{D_2}\left[ \eta _2\left( \phi \right) \right]}{\partial \nu _1} \left( \dfrac{\delta _1}{\delta _2}x \right) +\epsilon _2\mathcal{L} _{D_{1}^{2}}\left[ \eta _2\left( \phi \right) \right] \left( \dfrac{\delta _1}{\delta _2}x \right) +\epsilon _1\mathcal{R} _{D_{1}^{2}}\left[ \eta _2\left( \phi \right) \right] \left( \dfrac{\delta _1}{\delta _2}x \right) +\mathcal {O}\left( \epsilon _{1}^{1+\beta}\right)+\mathcal {O}\left(\epsilon _{2}^{1+\beta} \right),\\
        &  \frac{\partial \mathcal{S} _{D_{\epsilon _1}}\left[ \tilde{\eta}_1\left( \varphi \right) \right]}{\partial \tilde{\nu}_2} \left( \dfrac{\delta _2}{\delta _1}\tilde{x} \right)  \nonumber \\
    	= & \frac{\partial \mathcal{S} _{D_1}\left[ \eta _1\left( \varphi \right) \right]}{\partial \nu _2} \left( \dfrac{\delta _2}{\delta _1}x \right) +\epsilon _1\mathcal{L} _{D_{2}^{1}}\left[ \eta _2\left( \phi \right) \right] \left( \dfrac{\delta _2}{\delta _1}x \right) +\epsilon _2\mathcal{R} _{D_{2}^{1}}\left[ \eta _2\left( \phi \right) \right] \left( \dfrac{\delta _2}{\delta _1}x \right) +\mathcal {O}\left( \epsilon _{1}^{1+\beta} \right) +\mathcal {O}\left( \epsilon _{2}^{1+\beta} \right) ,
    \end{align*}
    where
    \begin{align*}
    	\mathcal{L} _{D_{1}^{2}}\left[ \eta _2\left( \phi \right) \right] & =\frac{\partial \mathcal{D} _{D_2}\left[ h_2 \eta _2\left( \phi \right)  \right]}{\partial \nu _1}-\frac{\partial \mathcal{S} _{D_2}\left[ \tau _2h_2\eta _2\left( \phi \right) \right]}{\partial \nu _1} ,\\
    	\mathcal{R} _{D_{1}^{2}}\left[ \eta _2\left( \phi \right) \right] & =-h_{1}^{'}\left( t \right) \frac{\partial \mathcal{S} _{D_2}\left[ \eta _2\left( \phi \right) \right]}{\partial T_1} +h_1\left( t \right) \left< D^2 \mathcal{S} _{D_2}\left[ \eta _2\left( \phi \right) \right]  \nu _1,\nu _1 \right> ,\\
        \mathcal{L} _{D_{2}^{1}}\left[ \eta _2\left( \phi \right) \right] & =\frac{\partial \mathcal{D} _{D_1}\left[ h_1 \eta _1\left( \varphi \right) \right] }{\partial \nu _2}-\frac{\partial \mathcal{S} _{D_1}\left[ \tau _1h_1\eta _1\left( \varphi \right) \right]}{\partial \nu _2}, \nonumber \\
    	\mathcal{R} _{D_{2}^{1}}\left[ \eta _2\left( \phi \right) \right] & =-h_{2}^{'}\left( t \right) \frac{\partial \mathcal{S} _{D_1}\left[ \eta _1\left( \phi \right) \right]}{\partial T_2} +h_2\left( t \right) \left< D^2 \mathcal{S} _{D_1}\left[ \eta _1\left( \phi \right) \right]  \nu _2,\nu _2 \right> .
    \end{align*}
    \end{thm}
    \begin{proof}
    First, we derive the asymptotic expansion of $\dfrac{\partial \mathcal{S} _{D_{\epsilon _2}}\left[ \tilde{\eta}_2\left( \phi \right) \right]}{\partial \tilde{\nu}_1} \left( \dfrac{\delta _1}{\delta _2}\tilde{x} \right)$ for $\tilde{x}\in \partial D_{\epsilon _1}$. When $\tilde{y}=y+\epsilon _2 h_2\left( y \right) \nu _2\left( y \right) \in \partial D_{\epsilon _2}$ with $y \in \partial D_2$, we get
    \begin{equation*}
    	\label{|x-y|_expansion}
    	\ln \left| x-\tilde{y} \right|=\ln \left| x-y \right|-\epsilon _2h_2\left( y \right) \frac{\left< x-y,\nu _2\left( y \right) \right>}{\left| x-y \right|^2}+\mathcal {O}\left( \epsilon _{2}^{2} \right).
    \end{equation*}
    Then, we have the asymptotic expansion of the single-layer potential $\mathcal{S} _{D_{\epsilon _2}}$ with respect to $\mathcal{S} _{D_2}$
    \begin{align}
    	\label{S_{D_e2}_expansion}
    	\mathcal{S} _{D_{\epsilon _2}}\left[ \tilde{\eta}_2\left( \phi \right) \right] \left( x \right)  & =\frac{1}{2\pi}\int_{\partial D_{\epsilon _2}}{\ln \left| x-\tilde{y} \right|\tilde{\eta}_2\left( \phi \right) \left( \tilde{y} \right) \d\sigma _{\epsilon _2}\left( \tilde{y} \right)}  \nonumber \\
    	& =\mathcal{S} _{D_2}\left[ \eta _2\left( \phi \right) \right] \left( x \right) +\epsilon _2\left[ \mathcal{D} _{D_2}\left[ h_2 \eta _2\left( \phi \right) \right] \left( x \right) -\mathcal{S} _{D_2}\left[ \tau _2h_2\eta _2\left( \phi \right) \right] \left( x \right) \right] +\mathcal {O}\left( \epsilon _{2}^{2} \right) .
    \end{align}

    From Lemma \ref{lemma4.1} and the Taylor expansion, we get
    \begin{align}
    	& \quad \,\,  \frac{\partial \mathcal{S} _{D_{\epsilon _2}}\left[ \tilde{\eta}_2\left( \phi \right) \right]}{\partial \tilde{\nu}_1} \left( \dfrac{\delta _1}{\delta _2}\tilde{x} \right) \nonumber \\
        & =\tilde{\nu}_1\left( \tilde{x} \right) \cdot \bigtriangledown \mathcal{S} _{D_{\epsilon _2}}\left[ \tilde{\eta}_2\left( \phi \right) \right] \left( \dfrac{\delta _1}{\delta _2}\tilde{x} \right)  \nonumber \\
    	& =\left[ \nu _1\left( x \right) -\epsilon _1h_{1}^{'}\left( t \right) T_1\left( x \right) \right]  \cdot \left[ \bigtriangledown \mathcal{S} _{D_{\epsilon _2}}\left[ \tilde{\eta}_2\left( \phi \right) \right] +\epsilon _1h_1\left( t \right) \sum_{j=1}^2{\partial _j\left( \bigtriangledown \mathcal{S} _{D_{\epsilon _2}}\left[ \tilde{\eta}_2\left( \phi \right) \right] \right) \nu _{1j}} \right] \left( \dfrac{\delta _1}{\delta _2}x \right) +\mathcal {O}\left( \epsilon _{1}^{1+\beta} \right)
    	\nonumber \\
    	& =\frac{\partial \mathcal{S} _{D_{\epsilon _2}}\left[ \tilde{\eta}_2\left( \phi \right) \right]}{\partial \nu _1} \left( \dfrac{\delta _1}{\delta _2}x \right)  +\epsilon _1\left[ -h_{1}^{'}\left( t \right) \frac{\partial \mathcal{S} _{D_{\epsilon _2}}\left[ \tilde{\eta}_2\left( \phi \right) \right]}{\partial T_1} +h_1\left( t \right) \left< D^2 \mathcal{S} _{D_{\epsilon _2}}\left[ \tilde{\eta}_2\left( \phi \right) \right]  \nu _1,\nu _1 \right> \right] \left( \dfrac{\delta _1}{\delta _2}x \right)
    	+\mathcal {O}\left( \epsilon _{1}^{1+\beta} \right).\nonumber
    \end{align}
    From (\ref{S_{D_e2}_expansion}), it is calculated that
    \begin{flalign*}
    	&\ \frac{\partial \mathcal{S} _{D_{\epsilon _2}}\left[ \tilde{\eta}_2\left( \phi \right) \right]}{\partial \nu _1} \left( \dfrac{\delta _1}{\delta _2}x \right) =  \frac{\partial \mathcal{S} _{D_2}\left[ \eta _2\left( \phi \right) \right]}{\partial \nu _1} \left( \dfrac{\delta _1}{\delta _2}x \right)  +\epsilon _2\left[ \frac{\partial \mathcal{D} _{D_2}\left[ h_2 \eta _2\left( \phi \right)  \right]}{\partial \nu _1}-\frac{\partial \mathcal{S} _{D_2}\left[ \tau _2h_2\eta _2\left( \phi \right) \right]}{\partial \nu _1} \right] \left( \dfrac{\delta _1}{\delta _2}x \right) +\mathcal {O}\left( \epsilon _{2}^{2} \right), & \\
    	&\ \frac{\partial \mathcal{S} _{D_{\epsilon _2}}\left[ \tilde{\eta}_2\left( \phi \right) \right]}{\partial T_1} \left( \dfrac{\delta _1}{\delta _2}x \right) =  \frac{\partial \mathcal{S} _{D_2}\left[ \eta _2\left( \phi \right) \right]}{\partial T_1} \left( \dfrac{\delta _1}{\delta _2}x \right) +\epsilon _2\left[ \frac{\partial \mathcal{D} _{D_2}\left[ h_2 \eta _2\left( \phi \right)  \right]}{\partial T_1}-\frac{\partial \mathcal{S} _{D_2}\left[ \tau _2h_2\eta _2\left( \phi \right) \right]}{\partial T_1} \right] \left( \dfrac{\delta _1}{\delta _2} x \right) + \mathcal {O}\left( \epsilon _{2}^{2} \right), & \\
    	&\ \left< D^2 \mathcal{S} _{D_{\epsilon _2}}\left[ \tilde{\eta}_2\left( \phi \right) \right] \nu _1,\nu _1 \right> \left( \dfrac{\delta _1}{\delta _2} x \right)
    	= \left< D^2 \mathcal{S} _{D_2}\left[ \eta _2\left( \phi \right) \right] \nu _1,\nu _1 \right> \left( \dfrac{\delta _1}{\delta _2} x \right) & \\
    	&\ \qquad \qquad \qquad \qquad \qquad \qquad \qquad \ \ +\epsilon _2\left< D^2 \left( \mathcal{D} _{D_2}\left[ h_2 \eta _2\left( \phi \right) \right] -\mathcal{S} _{D_2}\left[ \tau _2h_2\eta _2\left( \phi \right) \right] \right) \nu _1,\nu _1 \right> \left( \dfrac{\delta _1}{\delta _2}x \right) +\mathcal {O}\left( \epsilon _{2}^{2} \right) .&
    \end{flalign*}
    Finally, it is clearly obtained that
    \begin{align*}
    	& \dfrac{\partial \mathcal{S} _{D_{\epsilon _2}}\left[ \tilde{\eta}_2\left( \phi \right) \right]}{\partial \tilde{\nu}_1} \left( \dfrac{\delta _1}{\delta _2}\tilde{x} \right)  \\
    = & \dfrac{\partial \mathcal{S} _{D_2}\left[ \eta _2\left( \phi \right) \right]}{\partial \nu _1} \left( \dfrac{\delta _1}{\delta _2}x \right)  +\epsilon _2\left[ \frac{\partial \mathcal{D} _{D_2}\left[ h_2 \eta _2\left( \phi \right)  \right]}{\partial \nu _1}-\frac{\partial \mathcal{S} _{D_2}\left[ \tau _2h_2\eta _2\left( \phi \right) \right]}{\partial \nu _1} \right] \left( \dfrac{\delta _1}{\delta _2}x \right)  \\
    	& +\epsilon _1\left[ -h_{1}^{'}\left( t \right) \frac{\partial \mathcal{S} _{D_2}\left[ \eta _2\left( \phi \right) \right]}{\partial T_1} +h_1\left( t \right) \left< D^2 \mathcal{S} _{D_2}\left[ \eta _2\left( \phi \right) \right]  \nu _1,\nu _1 \right> \right] \left( \dfrac{\delta _1}{\delta _2}x \right) +\mathcal {O}\left( \epsilon _{1}^{1+\beta}\right)+\mathcal {O}\left(\epsilon _{2}^{1+\beta} \right).
    \end{align*}
    Through a completely consistent analysis, we can obtain the dual asymptotic expansion of  $\dfrac{\partial \mathcal{S} _{D_{\epsilon _1}}\left[ \tilde{\eta}_1\left( \varphi \right) \right]}{\partial \tilde{\nu}_2} \left( \dfrac{\delta _2}{\delta _1}\tilde{x} \right) $ for $\tilde{x}\in \partial D_{\epsilon _2}$. From this, the statement follows.
    \end{proof}

    \begin{cor}
    \label{theorem4.3}
    For the right-hand sides $\dfrac{\partial H\left( \delta _1\tilde{x} \right)}{\partial \tilde{\nu}_1}$, $ \tilde{x}\in \partial D_{\epsilon _1}$ and $\dfrac{\partial H\left( \delta _2\tilde{x} \right)}{\partial \tilde{\nu}_2}$, $\tilde{x}\in \partial D_{\epsilon _2}$, we have the following asymptotic expansions hold:
    \begin{align*}
    	\frac{\partial H\left( \delta _1\tilde{x} \right)}{\partial \tilde{\nu}_1} &= \frac{\partial H\left( \delta _1x \right)}{\partial \nu _1}+\epsilon _1\mathcal{R} _{H_1}\left( \delta _1x \right) +\mathcal {O}\left( \epsilon _{1}^{1+\beta} \right), \\
    	\frac{\partial H\left( \delta _2\tilde{x} \right)}{\partial \tilde{\nu}_2} &= \frac{\partial H\left( \delta _2x \right)}{\partial \nu _2}+\epsilon _2\mathcal{R} _{H_2}\left( \delta _2x \right) +\mathcal {O}\left( \epsilon _{2}^{1+\beta} \right),  	
    \end{align*}
    where
    \begin{align*}
    	\mathcal{R} _{H_1}& =-h_{1}^{'}\left( t \right) \frac{\partial H}{T_1}+h_1\left( t \right) \left< D^2H\nu _1,\nu _1 \right>  ,\\
        \mathcal{R} _{H_2}& =-h_{2}^{'}\left( t \right) \frac{\partial H}{T_2}+h_2\left( t \right) \left< D^2H\nu _2,\nu _2 \right> .
    \end{align*}
    \end{cor}

    Based on the asymptotic analysis of Theorems \ref{theorem4.1},   \ref{theorem4.2} and \ref{theorem4.3}, we can obtain that the transmission conditions (\ref{tbc5.2}) are equivalent to
    \begin{equation}
    	\label{tbc5.3}
    	\begin{cases}
    		\left( z\left( \omega \right) I-\mathcal{K} _{D_1}^{*}-\epsilon _1\mathcal{K} _{D_1}^{\left( 1 \right)} \right) \left[ \eta _1\left( \varphi \right) \right] \left( x \right) -\dfrac{\delta _2}{\delta _1}\left( \dfrac{\partial \mathcal{S} _{D_2}}{\partial \nu _1}+\epsilon _2\mathcal{L} _{D_{1}^{2}}+\epsilon _1\mathcal{R} _{D_{1}^{2}} \right) \left[ \eta _2\left( \phi \right) \right] \left( \dfrac{\delta _1}{\delta _2}x \right) \vspace{1ex} \\
    		\qquad \qquad \qquad \qquad \qquad =\dfrac{1}{\delta _1}\left( \dfrac{\partial H}{\partial \nu _1}+\epsilon _1\mathcal{R} _{H_1} \right)\left( \delta _1x \right) +\mathcal {O}\left( \epsilon _{1}^{1+\beta} \right) +\mathcal {O}\left( \epsilon _{2}^{1+\beta} \right) , \quad x\in \partial D_1,\vspace{1ex} \\
    		\dfrac{\delta _1}{\delta _2}\left( \dfrac{\partial \mathcal{S} _{D_1}}{\partial \nu _2}+\epsilon _1\mathcal{L} _{D_{2}^{1}}+\epsilon _2\mathcal{R} _{D_{2}^{1}} \right) \left[ \eta _1\left( \varphi \right) \right] \left( \dfrac{\delta _2}{\delta _1}x \right) +\left( z\left( \omega \right) I+\mathcal{K} _{D_1}^{*}+\epsilon _2\mathcal{K} _{D_2}^{\left( 1 \right)} \right) \left[ \eta _2\left( \phi \right) \right] \left( x \right)\vspace{1ex} \\
    		\qquad \qquad \qquad \qquad \qquad =-\dfrac{1}{\delta _2}\left( \dfrac{\partial H}{\partial \nu _2}+\epsilon _2\mathcal{R} _{H_2} \right)\left( \delta _2x \right) +\mathcal {O}\left( \epsilon _{1}^{1+\beta} \right) +\mathcal {O}\left( \epsilon _{2}^{1+\beta} \right) , \quad x\in \partial D_2.\vspace{1ex} \\
    	\end{cases}
    \end{equation}
    The equations \eqref{tbc5.3} can be written simply as
    \begin{equation*}
    	\left[ z\left( \omega \right) I+\mathbb{K} _{\delta_1,\delta_2}^{*} +\epsilon _1\mathbb{K} _{\delta_1,\delta_2,1}^{*} +\epsilon _2\mathbb{K} _{\delta_1,\delta_2,2}^{*}  \right] \Phi  =F_{\delta_1,\delta_2 ,\epsilon_1, \epsilon_2}+\mathcal {O}\left( \epsilon _{1}^{1+\beta}\right) +\mathcal {O}\left(\epsilon _{2}^{1+\beta} \right),
    \end{equation*}
    where
    \begin{equation}
        \label{three-matrix-1}
    	\mathbb{K} _{\delta_1,\delta_2}^{*} =\left[ \begin{matrix}
    		-\mathcal{K} _{D_1}^{*}&		-\dfrac{\delta _2}{\delta _1}\dfrac{\partial \mathcal{S} _{D_2}\left( \frac{\delta _1}{\delta _2} \right)}{\partial \nu _1} \vspace{1ex}\\
    		\dfrac{\delta _1}{\delta _2}\dfrac{\partial \mathcal{S} _{D_1}\left( \frac{\delta _2}{\delta _1} \right)}{\partial \nu _2}&		\mathcal{K} _{D_2}^{*}\\
    	\end{matrix} \right],
    \end{equation}
    \begin{equation}
        \label{three-matrix-2}
    	\mathbb{K} _{\delta_1,\delta_2,1}^{*}
    	=\left[ \begin{matrix}
    		-\mathcal{K} _{D_1}^{\left( 1 \right)}&		-\dfrac{\delta _2}{\delta _1}\mathcal{R} _{D_{1}^{2}} \vspace{2ex} \\
    		\dfrac{\delta _1}{\delta _2}\mathcal{L} _{D_{2}^{1}}&		0\\
    	\end{matrix} \right] , \quad
    	\mathbb{K} _{\delta_1,\delta_2,2}^{*}
    	=\left[ \begin{matrix}
    		0&		-\dfrac{\delta _2}{\delta _1}\mathcal{L} _{D_{1}^{2}} \vspace{2ex} \\
    		\dfrac{\delta _1}{\delta _2}\mathcal{R} _{D_{2}^{1}}&		\mathcal{K} _{D_2}^{\left( 1 \right)}\\
    	\end{matrix} \right] ,
    \end{equation}
    \begin{equation*}
    	\Phi =\left[ \begin{array}{c}
    		\eta _1\left( \varphi \right) \vspace{2ex} \\
    		\eta _2\left( \phi \right)\\
    	\end{array} \right] , \quad
    	F_{\delta_1,\delta_2 ,\epsilon_1, \epsilon_2}=\left[ \begin{array}{c}
    		\dfrac{1}{\delta _1}\left( \dfrac{\partial H}{\partial \nu _1}+\epsilon _1\mathcal{R} _{H_1} \right) \left( \delta _1x \right)\vspace{1ex} \\
    		-\dfrac{1}{\delta _2}\left( \dfrac{\partial H}{\partial \nu _2}+\epsilon _2\mathcal{R} _{H_2} \right) \left( \delta _2x \right)\\
    	\end{array} \right].
    \end{equation*}
    Note that the meaning of the notation above is $\mathcal{S} _{D_1} \left( \frac{\delta _2}{\delta _1} \right)= \mathcal{S} _{D_1} \left[ \cdot \right] \left( \frac{\delta _2}{\delta _1} x \right)$, $\mathcal{S} _{D_2} \left( \frac{\delta _1}{\delta _2} \right)= \mathcal{S} _{D_2} \left[ \cdot \right] \left( \frac{\delta _1}{\delta _2} x \right)$.

\subsection{Proof of Theorem 2.1}
    Now, we will prove Theorem \ref{theorem2.1}. First, according to the properties of NP-type operator $\mathbb{K} _{\delta_1,\delta_2}^{*}$, it is easy to find the orthonormal eigenvectors $\varPhi _n \,( n=1,2,\cdots ) $ under the $\left<  \cdot, \cdot \right> _{\mathcal{H} ^2}$ inner product with respect to $\lambda _{n}$, i.e., $\varPhi _n$ satisfies $\left\| \varPhi _n \right\| _{\mathcal{H} ^2}=1$, $\left< \varPhi _i,\varPhi _j \right> _{\mathcal{H} ^2}=\delta _{ij}$.
    \begin{proof}[Proof of Theorem 2.1]
    For $n=1,2,\cdots$, we can use standard perturbation theory to represent the perturbed eigenvalue and eigenvector as the following, respectively
    \begin{align*}
    	\tilde{\lambda}_n & = \lambda _n+\epsilon _1\lambda _{n}^{1}+\epsilon _2\lambda _{n}^{2}+\mathcal {O}\left( \epsilon _{1}^{1+\beta} \right) +\mathcal {O}\left( \epsilon _{2}^{1+\beta} \right) , \\
    	\tilde{\varPhi }_n & = \varPhi _n+\epsilon _1\varPhi _{n}^{1}+\epsilon _2\varPhi _{n}^{2}+\mathcal {O}\left( \epsilon _{1}^{1+\beta} \right) +\mathcal {O}\left( \epsilon _{2}^{1+\beta} \right).
    \end{align*}
    The main term clearly holds. Next we will focus on the first-order terms of $\epsilon _1$ and $\epsilon _2$ . For the first-order terms of $\epsilon _1$ and $\epsilon _2$, we obtain the following identities respectively
    \begin{align*}
    	\mathbb{K} _{\delta_1,\delta_2 }^{*} \varPhi _{n}^{1}+\mathbb{K} _{\delta_1,\delta_2 , 1}^{*} \varPhi _n & =\lambda _n\varPhi _{n}^{1}+\lambda _{n}^{1}\varPhi _n,\\
    	\mathbb{K} _{\delta_1,\delta_2 }^{*} \varPhi _{n}^{2}+\mathbb{K} _{\delta_1,\delta_2 , 2}^{*} \varPhi _n & =\lambda _n\varPhi _{n}^{2}+\lambda _{n}^{2}\varPhi _n.
    \end{align*}
    Making an inner product of $\left<  \cdot, \cdot \right> _{\mathcal{H} ^2}$ for both sides of the above two identities with respect to $\varPhi _n$, we obtain
    \begin{align*}
    	\lambda _{n}^{1} & =\left< \mathbb{K} _{\delta_1,\delta_2 , 1}^{*} \varPhi _n,\varPhi _n \right> _{\mathcal{H} ^2},\\
    	\lambda _{n}^{2} & =\left< \mathbb{K} _{\delta_1,\delta_2 , 2}^{*} \varPhi _n,\varPhi _n \right> _{\mathcal{H} ^2}.
    \end{align*}
    We can also calculate that $\varPhi _{n}^{1}$ and $\varPhi _{n}^{2}$. Therefore, the perturbed eigenvalues of the NP-type operator $ \mathbb{K} _{\epsilon ,\delta}^{*} $ can be represented by
    \begin{equation*}
    	\tilde{\lambda}_n = \lambda _n+\epsilon _1 \left< \mathbb{K} _{\delta_1,\delta_2 , 1}^{*} \varPhi _n,\varPhi _n \right> _{\mathcal{H} ^2}+\epsilon _2 \left< \mathbb{K} _{\delta_1,\delta_2 , 2}^{*} \varPhi _n,\varPhi _n \right> _{\mathcal{H} ^2}+\mathcal {O}\left( \epsilon _{1}^{1+\beta} \right) +\mathcal {O}\left( \epsilon _{2}^{1+\beta} \right) .
    \end{equation*}
    Combined with the Drude's model, the perturbed eigenvalues of $ \mathbb{K} _{\delta_1,\delta_2 ,\epsilon_1, \epsilon_2}^{*} $ are substituted into the first-order corrected resonance condition (\ref{first-order-resonance-condition}), then the proof is complete.
    \end{proof}

    The following corollary is a direct consequence of Theorem \ref{theorem2.1}. Therefore, its proof is omitted.

    \begin{cor}
    For concentric disk-shaped metallic nanoshell, the resonance frequency $\omega _{n\pm}$ is given by the following formula
    \begin{equation}
    	\label{2d_resonance_frequency}
    	\omega _{n\pm}\left( \delta _1 , \delta _2 \right)=\frac{\omega _p}{\sqrt{2}}\sqrt{1-2 \lambda _n}, \quad  n=1,2,\cdots.
    \end{equation}
    \end{cor}

    As is well known, the concentric disk is a very special geometric structure, and the eigenvalues of its NP-type operator can be computed explicitly. Therefore, (\ref{2d_resonance_frequency}) has a more specific representation, which will be further studied, as this is very important for guiding the practical applications of plasmon.

\subsection{Eigenvalues of the NP-type operator \texorpdfstring{$\mathbb{K} _{\delta _1 , \delta _2}^{*}$}{K*} }
    The main term of the asymptotic analyses of $\epsilon _1$ and $\epsilon _2$ is defined on $L^{2}_{0}\left( \partial D_{1} \right) \times L^{2}_{0}\left( \partial D_{2} \right) $. To investigate the property of the NP-type operator $\mathbb{K} _{\delta _1 , \delta _2}^{*}$, we first give the following lemmas \cite{ammari2013mathematical}.
    \begin{lem}
    \label{lemma4.5}
    Suppose that $D \subset \mathbb{R} ^2 $ is a disk with radius $r_0$, then
    \begin{equation*}
    	    \mathcal{K} _D^{*}\left[ \mathrm{e}^{in\theta} \right] \left( x \right) =0, \quad  \forall n\ne 0.
    \end{equation*}
    \end{lem}

    \begin{lem}
    \label{lemma4.6}
    If a disk $D$ of radius $r_0$ is centered at the origin, for each integer $n$, we have
    \begin{equation*}
    	\frac{\partial \mathcal{S} _D\left[ \mathrm{e}^{\mathrm{i}n\theta} \right]}{\partial \nu} \left( x \right) =\begin{cases}
    		-\dfrac{1}{2}\left( \dfrac{r}{r_0} \right) ^{|n|-1}\mathrm{e}^{\mathrm{i}n\theta}, \quad |x|=r<r_0,\vspace{2ex} \\
    		-\dfrac{1}{2}\left( \dfrac{r_0}{r} \right) ^{|n|+1}\mathrm{e}^{\mathrm{i}n\theta}, \quad |x|=r>r_0.\\
    	\end{cases}
    \end{equation*}
    \end{lem}

    \begin{thm}
    \label{theorem4.11}
    The interfaces $\partial D_1 $, $\partial D_2 $ of bounded domains $D_1$, $D_2 \subset \mathbb{R} ^2 $ are circles with radius $r_1$, $r_2$ respectively, then we derive the eigenvalues of the NP-type operator $\mathbb{K} _{\delta _1 , \delta _2}^{*}$ to be
    \begin{equation*}
    	\lambda _{n}=\pm \frac{1}{2}\left( \dfrac{\delta _1}{\delta _2} \right) ^{|n|}\left( \dfrac{r _1}{r _2} \right) ^{|n|},\quad  n=1,2,\cdots
    \end{equation*}
    and the corresponding eigenvectors are
    \begin{equation*}
        \left[
    		\mathrm{e}^{\pm \mathrm{i}n\theta} \quad
    		\pm \frac{\delta _1}{\delta _2}\frac{r _1}{r _2}\mathrm{e}^{\pm \mathrm{i}n\theta}
        \right]^\mathrm{T}
    , \quad n=1,2,\cdots
    \end{equation*}
    \end{thm}
    \begin{proof}
    Let the eigenvector of $\mathbb{K} _{\delta _1 , \delta _2}^{*}$ be
    $\left[
    	a\mathrm{e}^{\pm \mathrm{i}n\theta} \quad
    	b\mathrm{e}^{\pm \mathrm{i}n\theta}
     \right]^\mathrm{T}$, where $a$ and $b$ are undetermined coefficients. Let $x \in \partial D_1=\left\{ x~|~ |x|=r_1 \right\}$ , $ y \in \partial D_2=\left\{ y~|~|y|=r_2 \right\} $, then $\frac{\delta _1}{\delta _2} r_1<r_2$. From Lemma \ref{lemma4.6}, we obtain
    \begin{equation*}
    	\frac{\partial \mathcal{S} _{D_2}\left( \frac{\delta _1}{\delta _2} \right)\left[ \mathrm{e}^{\mathrm{i}n\theta} \right]}{\partial \nu _1}  \left( x \right)  = -\dfrac{1}{2}\left( \dfrac{\delta _1}{\delta _2} \right) ^{|n|}\left( \dfrac{r _1}{r _2} \right) ^{|n|-1}\mathrm{e}^{\mathrm{i}n\theta}\left( x \right),
    \end{equation*}
    next, let $x \in \partial D_2=\left\{ x~|~|x|=r_2 \right\}$, $y \in \partial D_1=\left\{ y~|~|y|=r_1 \right\} $, then $\frac{\delta _2}{\delta _1} r_2>r_1$. From Lemma \ref{lemma4.6} again, we obtain
    \begin{equation*}
    	\frac{\partial \mathcal{S} _{D_1}\left( \frac{\delta _2}{\delta _1} \right)\left[ \mathrm{e}^{\mathrm{i}n\theta} \right]}{\partial \nu _2}  \left( x \right)=\dfrac{1}{2}\left( \frac{\delta _2}{\delta _1} \right) ^{|n|}\left( \frac{r _2}{r _1} \right) ^{|n|+1}\mathrm{e}^{\mathrm{i}n\theta}\left( x \right).
    \end{equation*}
    Utilizing Lemma \ref{lemma4.5}, the eigenvector satisfies the following identity
    \begin{equation*}
    	\left[ \begin{matrix}
    		0 &	\dfrac{1}{2}\left( \dfrac{\delta _1}{\delta _2} \right) ^{|n|-1}\left( \dfrac{r _1}{r _2} \right) ^{|n|-1} \vspace{2ex} \\
    		\dfrac{1}{2}\left( \dfrac{\delta _1}{\delta _2} \right) ^{|n|+1}\left( \dfrac{r _1}{r _2} \right) ^{|n|+1} &	0\\
    	\end{matrix} \right]
    	\left[ \begin{array}{c}
    		a\mathrm{e}^{\pm \mathrm{i}n\theta} \vspace{2ex} \\
    		b\mathrm{e}^{\pm \mathrm{i}n\theta}\\
    	\end{array} \right]
    	=\lambda _{n}
    	\left[ \begin{array}{c}
    		a\mathrm{e}^{\pm \mathrm{i}n\theta} \vspace{2ex} \\
    		b\mathrm{e}^{\pm \mathrm{i}n\theta}\\
    	\end{array} \right] .
    \end{equation*}
    It is easy to solve for the eigenvalues and eigenvectors of $\mathbb{K} _{\delta _1 , \delta _2}^{*}$.
    \end{proof}

    \begin{rem}
    From Theorem \ref{theorem4.11}, for concentric disk-shaped metallic nanoshell, the plasmon resonance frequency $\omega _{n \pm }\left( \delta _1 , \delta _2 \right)$ can be more specifically expressed as follows
    \begin{equation*}
    	\omega _{n\pm}\left( \delta _1 , \delta _2 \right)=\frac{\omega _p}{\sqrt{2}}\sqrt{1\pm \left( \frac{\delta _1}{\delta _2} \right) ^{|n|}\left( \frac{r _1}{r _2} \right) ^{|n|}}, \quad  n=1,2,\cdots.
    \end{equation*}
    \end{rem}

    \begin{rem}
     Resonance bandwidth $\Delta\omega$ is defined mathematically as
     \begin{equation*}
    	\Delta\omega = \omega_{\text{up}} - \omega_{\text{low}},
    \end{equation*}
    where $\omega_{\text{up}}$ and $\omega_{\text{low}}$ represent the upper and lower resonance frequency limits, respectively, which quantify the frequency range over which the resonant system exhibits a significant response.
    Considering the ratio of scale factors $\delta _1 $, $ \delta _2$ in the range of $\left( 0, r_2 / r_1 \right)$, the concentric disk-shaped metal nanoshell can be easily seen theoretically as $\Delta\omega$ varies in the range of $\left( 0, \omega _p \right)$. Thus, the resonance bandwidth reflects its energy dissipation properties as well as the versatility of the scale factor in applications such as sensing, imaging and molecular recognition.
    \end{rem}

\section{Numerical results}\label{sec:num}

In this section, we present a series of numerical experiments designed to validate the theoretical results, which demonstrate a strong agreement. By theoretical analysis, for concentric disc-shaped metallic nanoshells with normal perturbations, only the eigenvalues of the NP-type operator $\mathbb{K} _{\delta_1,\delta_2 ,\epsilon_1, \epsilon_2}^{*}$ are needed to obtain the plasmon resonance modes of the structure. First, we provide the values of the geometrical and physical parameters utilized in the finite element simulations before presenting the numerical results, setting the bulk plasmon frequency in the Drude's model $\omega _p = 9.06$ eV. Perturbations were applied to four distinct sets of shape functions:
\begin{itemize}
    \item \textbf{Set 1}:
    $h_1(t) = 0.5\cos(4t)$,
    $h_2(t) = -\sin(3t)$;

    \item \textbf{Set 2}:
    $h_1(t) = 0.5\sin(3t) + \sin(6t) - 0.5\sin(7t)$,
    $h_2(t) = -0.5\cos(4t) + \cos(5t) + 0.5\cos(7t)$;

    \item \textbf{Set 3}:
    $h_1(t) = 0.5\sin(t) + 0.5\sin(3t) - \sin(8t)$,
    $h_2(t) = -0.5\cos(t) - \cos(5t) + 0.5\cos(12t)$;

    \item \textbf{Set 4}:
    $h_1(t) = -\cos(6t)$,
    $h_2(t) = -\sin(4t)$.
\end{itemize}
The resulting nanostructures are illustrated in Figure \ref{figure_4shapefunction}.
\begin{figure}[htpb]
\centering
\subfigure[]{
\begin{minipage}[b]{0.23\textwidth}
\includegraphics[width=1\linewidth]{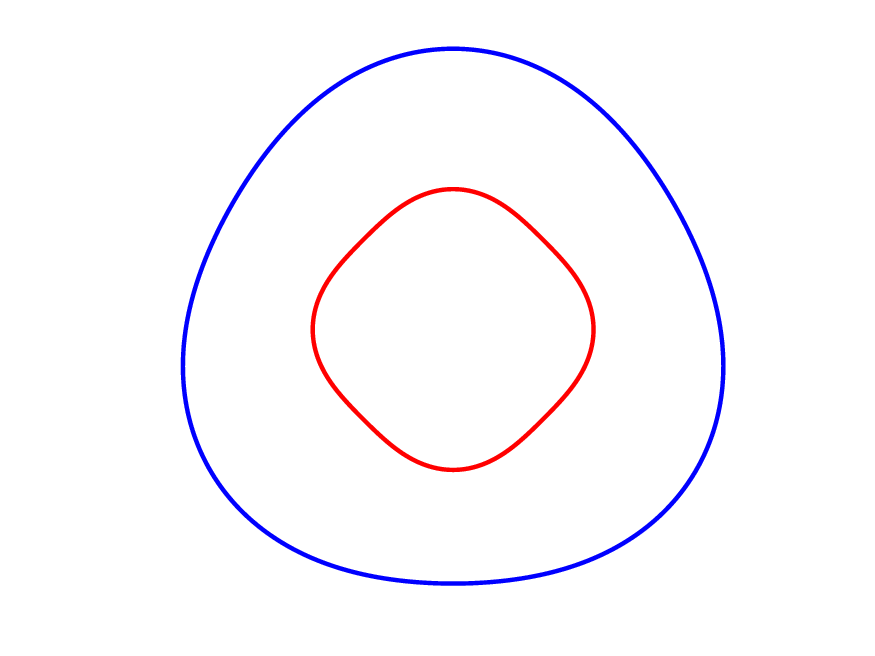}
\end{minipage}}
\subfigure[]{
\begin{minipage}[b]{0.23\textwidth}
\includegraphics[width=1\linewidth]{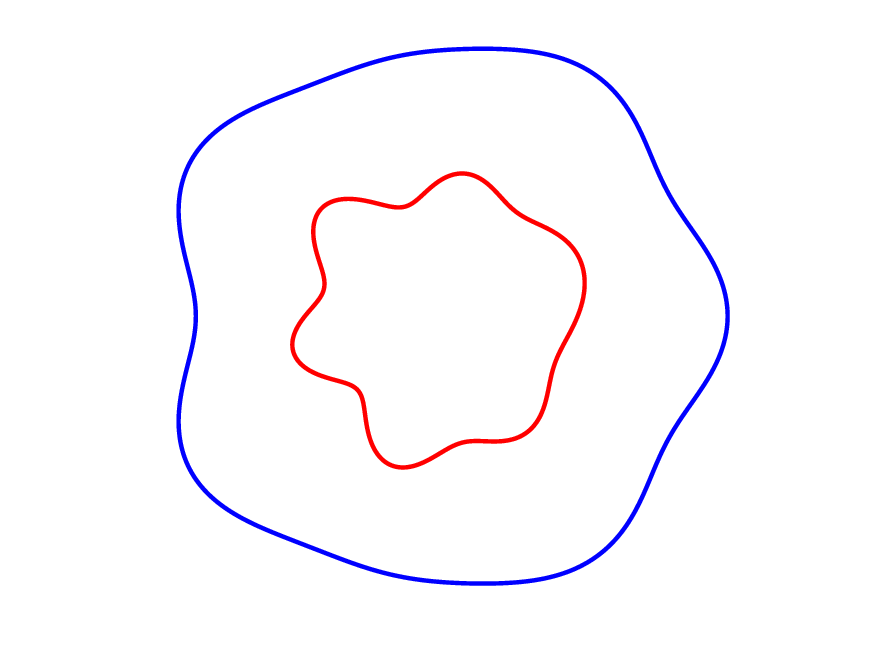}
\end{minipage}}
\subfigure[]{
\begin{minipage}[b]{0.23\textwidth}
\includegraphics[width=1\linewidth]{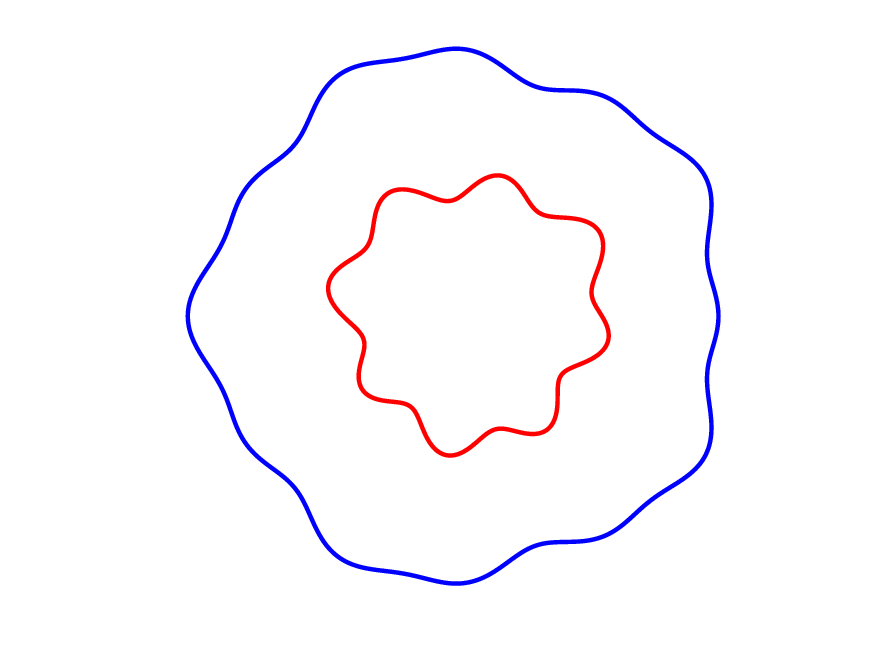}
\end{minipage}}
\subfigure[]{
\begin{minipage}[b]{0.23\textwidth}
\includegraphics[width=1\linewidth]{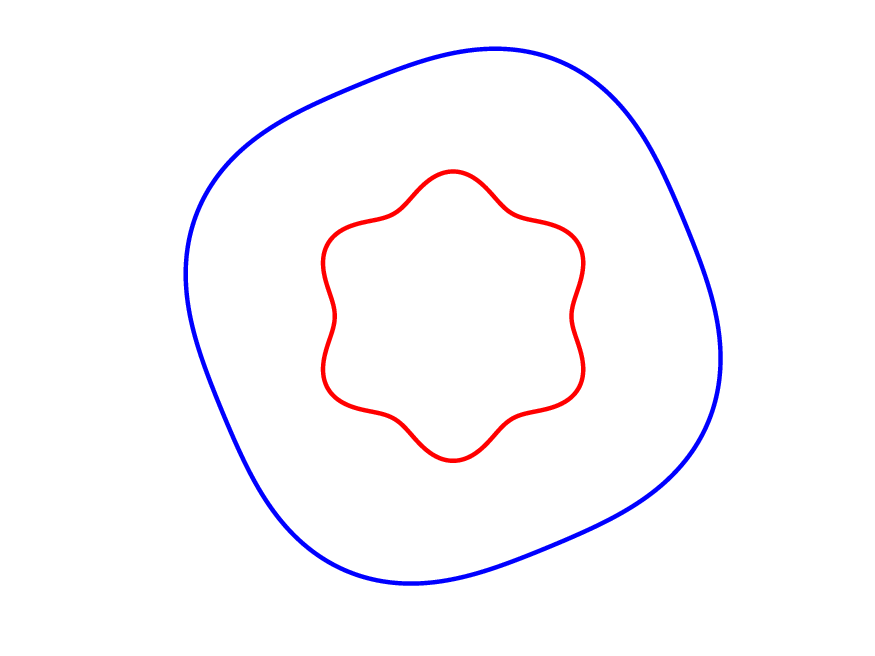}
\end{minipage}}
\caption{Schematic diagrams of nanostructures defined by four different shape functions. (a)-(d) correspond to Set 1-4, respectively.}
\label{figure_4shapefunction}
\end{figure}

Now, we consider the variation of the resonance frequency with the scale factor $\delta_1 \rightarrow 0$ as the metal nanoshell approaches the solid plasmon mode, and set the normal perturbation parameters $\epsilon _1 = \epsilon _2 = 0.01$. In the numerical experiments, we let the scale factor $\delta_1$ decrease gradually by $2^{m}$, (as $m\rightarrow -\infty$). Figure \ref{figure_solid} (a) and (b) present the numerical results for Set 1 of shape functions, while (c) and (d) correspond to Set 2 and Set 3, respectively.
Figure \ref{figure_solid} (a) displays the scattering intensity $|u^s|^2$ as a function of frequency (Hz) for various configurations, specifically with $\delta_1$  values of $\delta_1 = 2^{5},2^{4},2^{3}, 2^{2}, 2^{1}, 2^{0}, 2^{-1}, 2^{-2}$ and $2^{-3}$. Each curve is color-coded to represent different  $\delta_1$  values, allowing for a clear comparison of how varying $\delta_1$  influences the resonance frequencies. The peaks in the graph indicate the frequencies at which the system exhibits heightened response, suggesting the presence of resonant modes.
Figure \ref{figure_solid} (b)-(d) demonstrate the effect of the change in the scale $\delta_1$ of the metal nanocore on the hybridization strength under three different sets of shape functions when approaching the solid plasmon mode.  The discrete data points denote the resonance frequencies $\tilde{\omega}_{n\pm}$, while the black solid lines represent the resonance frequencies $\tilde{\omega}_{s,n}$ of the solid plasmon.
All cases reveal that as $\delta_1$ decays exponentially, the high-energy antibonding plasmon resonance frequencies $\tilde{\omega}_{n+}$ decrease while the low-energy bonding plasmon resonance frequencies $\tilde{\omega}_{n-}$ increase, with all asymptotically approaching the solid plasmon resonance frequency $\tilde{\omega}_{s,n}$.
It is evident that the resonance frequencies $\tilde{\omega}_{n\pm}$ exhibit distinct variations in both amplitude and quantity, which are intrinsically correlated with the inherent properties of the shape functions.
\begin{figure}[htpb]
\centering
\subfigure[]{
\begin{minipage}[b]{0.48\textwidth}
\includegraphics[width=1\linewidth]{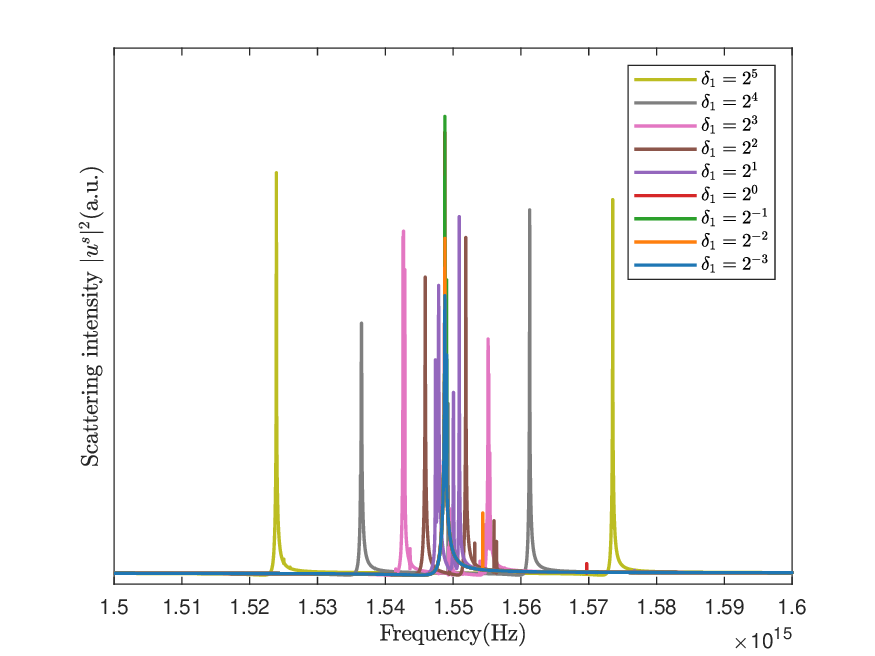}
\end{minipage}}
\subfigure[]{
\begin{minipage}[b]{0.48\textwidth}
\includegraphics[width=1\linewidth]{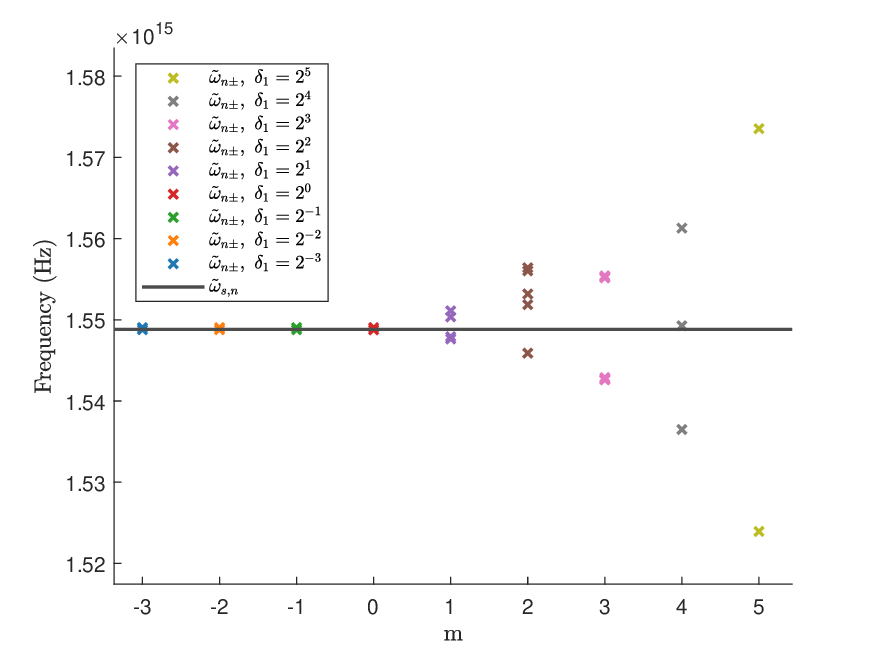}
\end{minipage}} \\
\subfigure[]{
\begin{minipage}[b]{0.48\textwidth}
\includegraphics[width=1\linewidth]{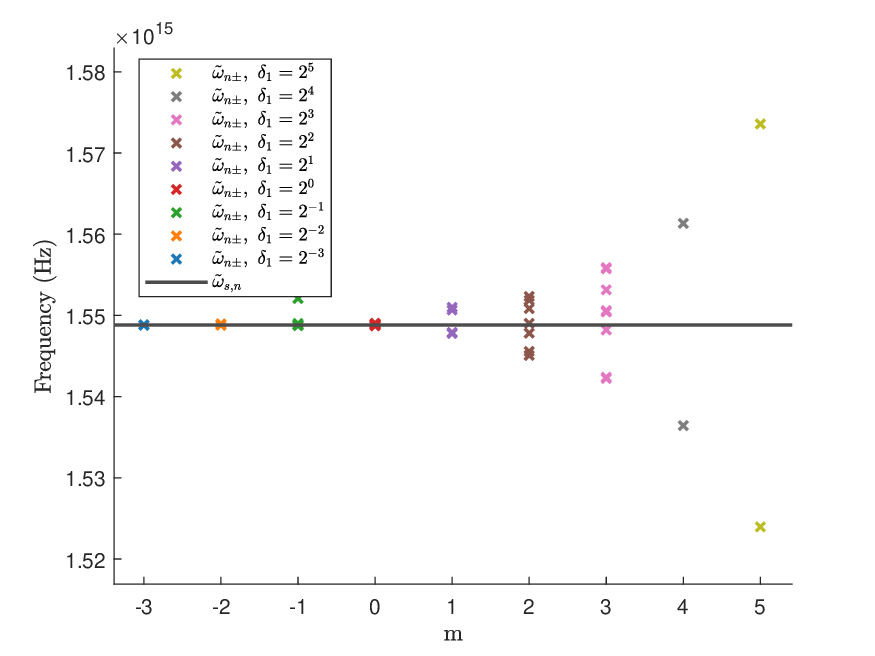}
\end{minipage}}
\subfigure[]{
\begin{minipage}[b]{0.48\textwidth}
\includegraphics[width=1\linewidth]{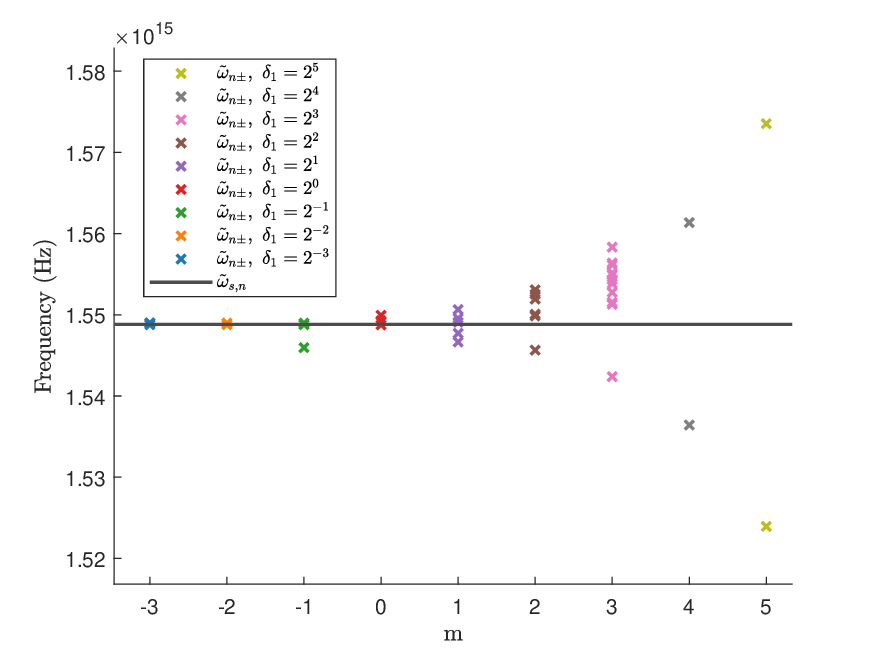}
\end{minipage}}
\caption{Hybridization behavior of metal nanoshells approaching the solid plasmon mode.  (a) Spectral response with respect to $\omega$ in Set 1 ($\delta_2$ fixed, $\delta_1 = 2^{5},2^{4},2^{3}, 2^{2}, 2^{1}, 2^{0}, 2^{-1}, 2^{-2}, 2^{-3}$). (b-d) Hybridization intensity of solid-cavity plasmon as a function of $\delta_1$ for Sets 1-3, respectively.}
\label{figure_solid}
\end{figure}

\begin{figure}[htpb]
\centering
\subfigure[]{
\begin{minipage}[b]{0.48\textwidth}
\includegraphics[width=1\linewidth]{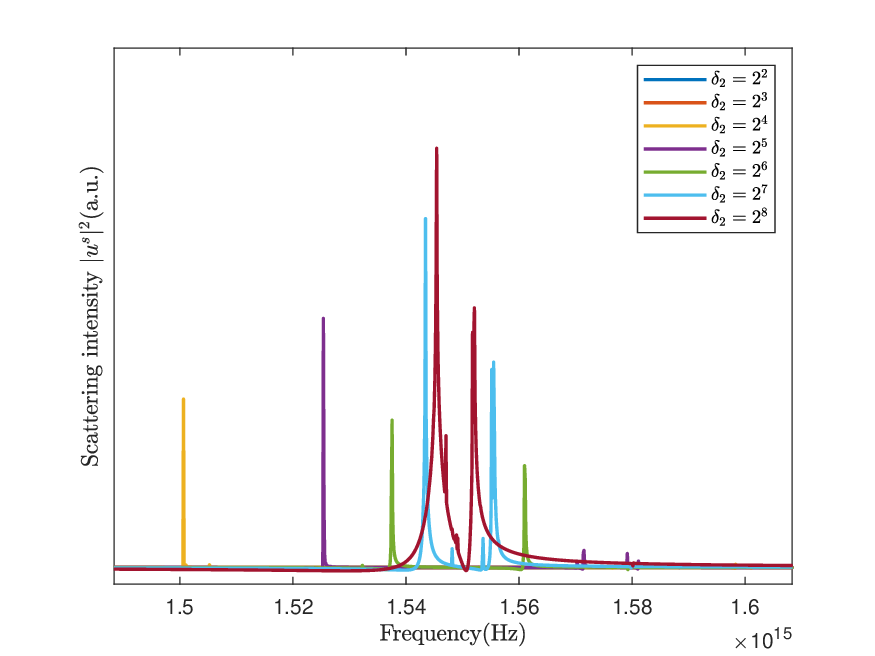}
\end{minipage}}
\subfigure[]{
\begin{minipage}[b]{0.48\textwidth}
\includegraphics[width=1\linewidth]{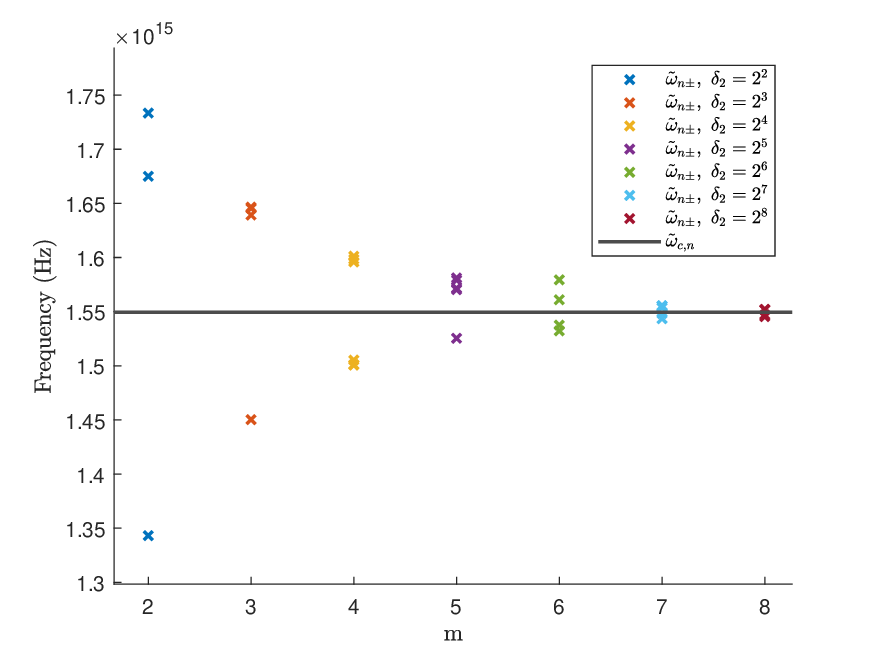}
\end{minipage}} \\
\subfigure[]{
\begin{minipage}[b]{0.48\textwidth}
\includegraphics[width=1\linewidth]{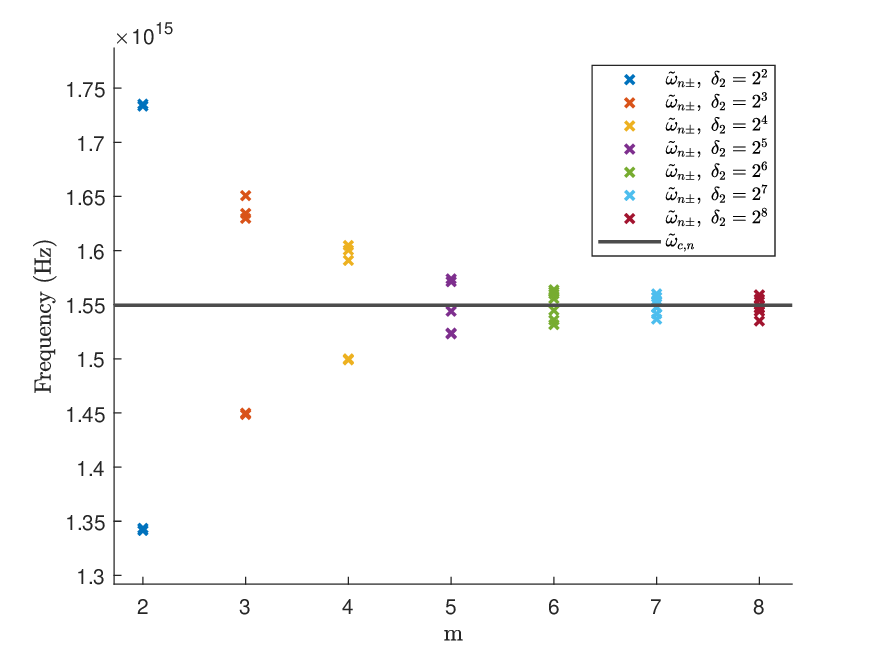}
\end{minipage}}
\subfigure[]{
\begin{minipage}[b]{0.48\textwidth}
\includegraphics[width=1\linewidth]{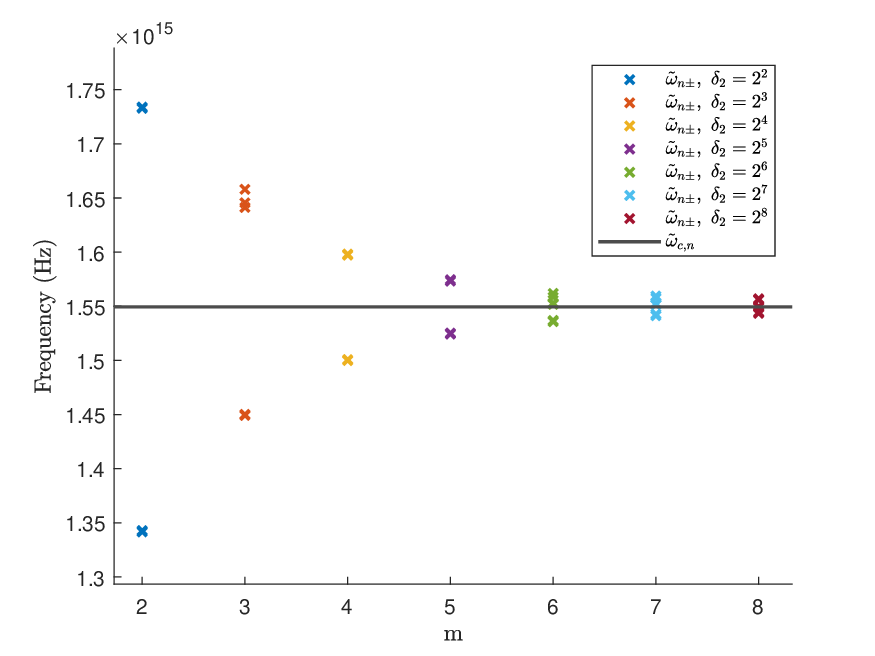}
\end{minipage}}
\caption{Hybridization behavior of metal nanoshells approaching the cavity plasmon mode.  (a) Spectral response with respect to $\omega$ in Set 1 ($\delta_1$ fixed, $\delta_2 = 2^{2}, 2^{3}, 2^{4}, 2^{5}, 2^{6}, 2^{7}, 2^{8}$). (b-d) Hybridization intensity of solid-cavity plasmon as a function of $\delta_2$ for Sets 1-3, respectively.}
\label{figure_cavity}
\end{figure}

Next, we considered the variation of the resonance frequency with the scale factor $\delta_2 \rightarrow \infty $ when the metal nanoshell approaches the cavity plasmon mode, and similarly set the normal perturbation parameter $\epsilon _1 = \epsilon _2 = 0.01$. Instead, we let the scale factor $\delta_1$ gradually increase by $ 2^m $, (as $m\rightarrow\infty$). Figure \ref{figure_cavity} (a) and (b) present the numerical results for Set 1 of shape functions, while (c) and (d) correspond to Set 2 and Set 3, respectively.
It has similar numerical results to the former, as can be seen in Figure \ref{figure_cavity} (b)-(d), where the effect of the variation of the scale factor $\delta_2$ on the hybridization strength makes the high-energy antibonding plasmon resonance frequencies $\tilde{\omega}_{n+}$ increasingly lower as the cavity plasmon mode is approached, while the low-energy bonding plasmon resonance frequencies  $\tilde{\omega}_{n-}$ grow higher and higher with the exponential growth of the scaling factor $\delta_2$, and all approach the cavity plasmon resonance frequency $\tilde{\omega}_{c,n}$ simultaneously and infinitely.
Similarly, the resonance frequencies
$\tilde{\omega}_{n\pm}$ exhibit distinct variations in both amplitude and quantity across different shape functions.

Finally, we systematically investigate how the perturbation intensity enriches the resonance frequencies under the Set 4 of shape functions. Figure \ref{figure_solid_per} presents the evolution patterns of resonance frequencies in metal nanoshells under three distinct perturbation intensities ($\epsilon_1 = \epsilon_2 =0.01, 0.05$, and $0.1$) as functions of the scale factor $\delta_1$, with panels (a)-(c) corresponding to each intensity respectively. This trend reflects the system's tendency to approach solid plasmon modes at diminishing $\delta_1$. Notably, for a fixed $m$ (corresponding to a constant scale factor  $\delta_1$), stronger perturbation intensities amplify the separation between the hybridized modes and the solid modes, particularly for $m=3$ and $4$. This observation suggests that enhanced perturbations not only induce frequency shifts but also promote more pronounced plasmon energy splitting and redistribution. The underlying mechanism can be attributed to the perturbation-driven energy redistribution within the system, which preferentially allocates energy to either higher or lower frequencies, exhibiting characteristic features of strong mode mixing.
Figure \ref{figure_cavity_per} presents the evolution of resonance frequencies in metal nanoshells under the aforementioned three perturbation intensities as a function of scale factor $\delta_2$, with panels (a)-(c) corresponding to each intensity level respectively, revealing similar trends.
These findings elucidate the synergistic interplay between perturbation intensity and geometric scaling in modulating plasmon resonance properties. This numerical study provides critical insights into the tunability of plasmon resonance through combined geometric and external perturbation control, offering potential applications in nanophotonic device engineering.

\begin{figure}[htpb]
\centering
\subfigure[]{
\begin{minipage}[b]{0.48\textwidth}
\includegraphics[width=1\linewidth]{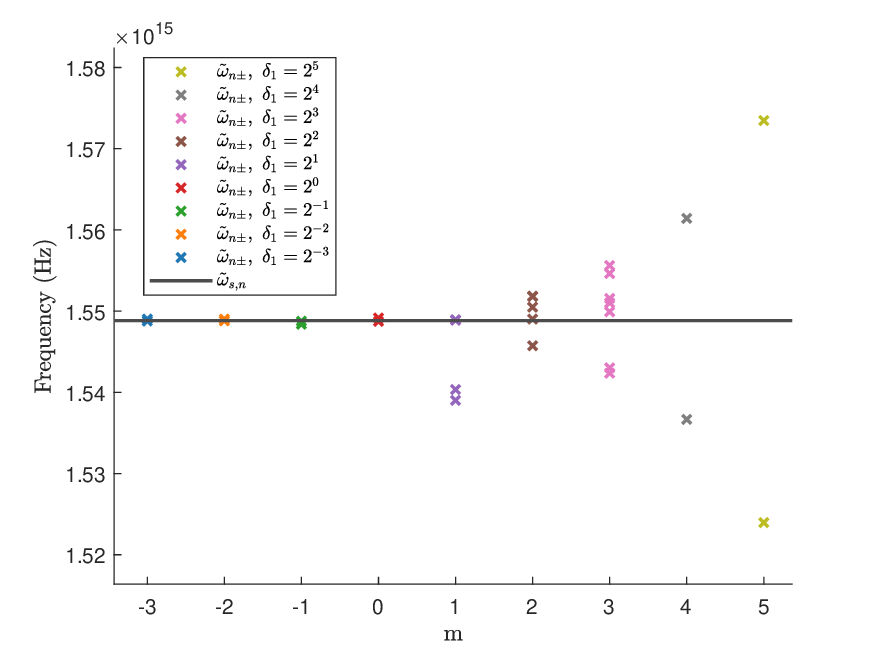}
\end{minipage}}
\subfigure[]{
\begin{minipage}[b]{0.48\textwidth}
\includegraphics[width=1\linewidth]{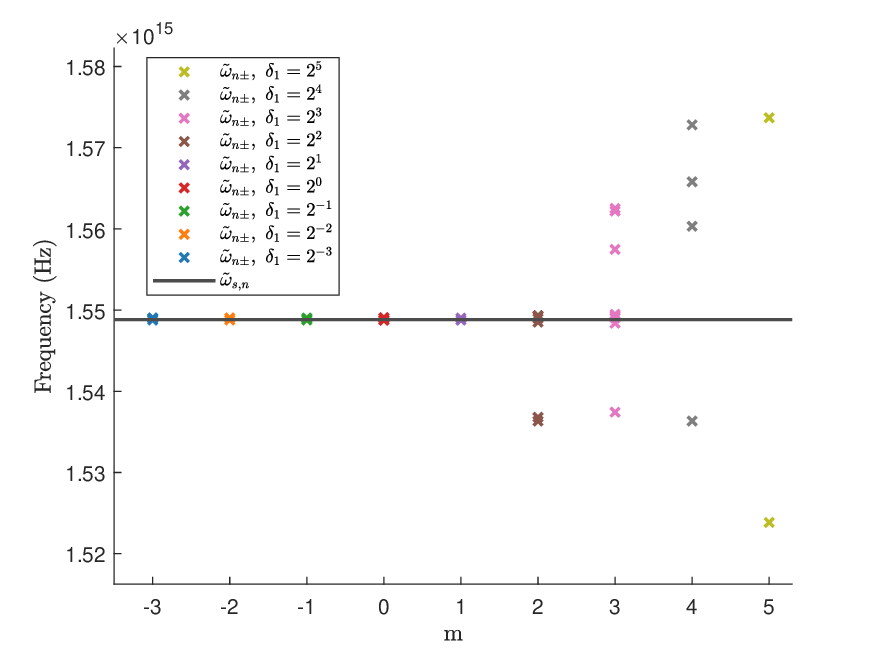}
\end{minipage}}
\subfigure[]{
\begin{minipage}[b]{0.48\textwidth}
\includegraphics[width=1\linewidth]{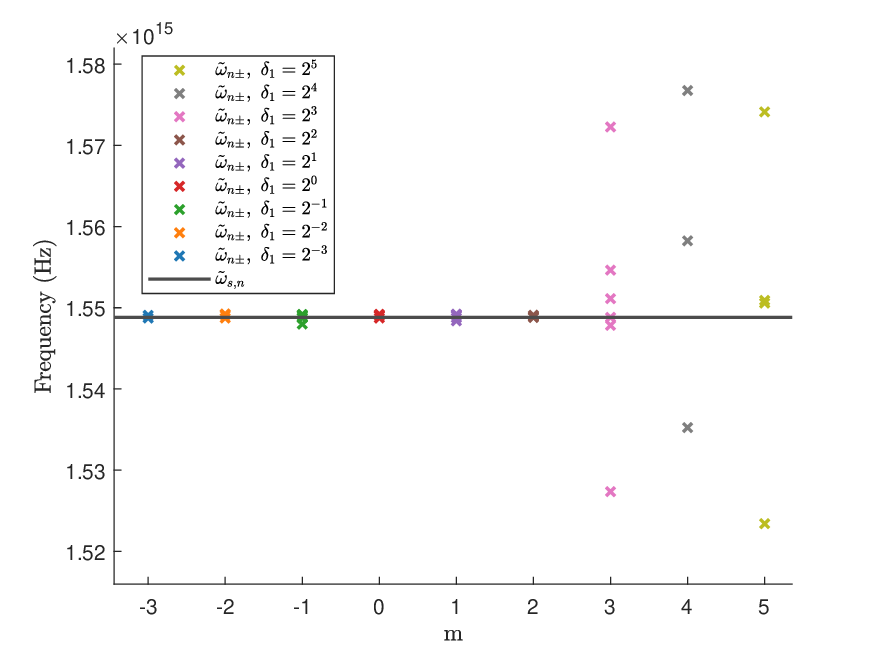}
\end{minipage}}
\caption{Hybridization intensity of solid-cavity plasmon depending on the scale factor $\delta_1$ with different perturbation parameters (a) $\epsilon_1 = \epsilon_2 = 0.01$; (b) $\epsilon_1 = \epsilon_2 = 0.05$; (c) $\epsilon_1 = \epsilon_2 = 0.1$.}
\label{figure_solid_per}
\end{figure}

\begin{figure}[htpb]
\centering
\subfigure[]{
\begin{minipage}[b]{0.48\textwidth}
\includegraphics[width=1\linewidth]{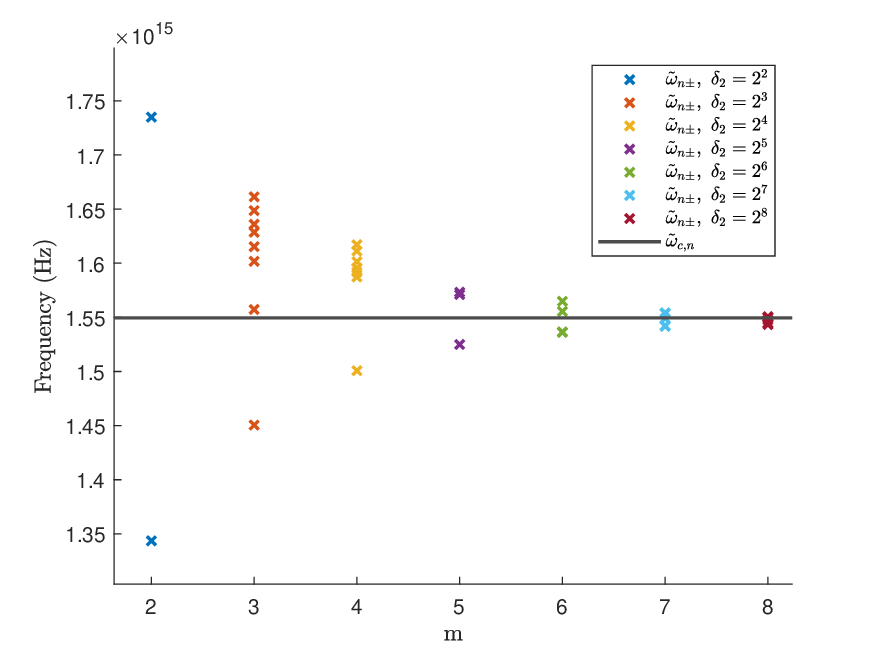}
\end{minipage}}
\subfigure[]{
\begin{minipage}[b]{0.48\textwidth}
\includegraphics[width=1\linewidth]{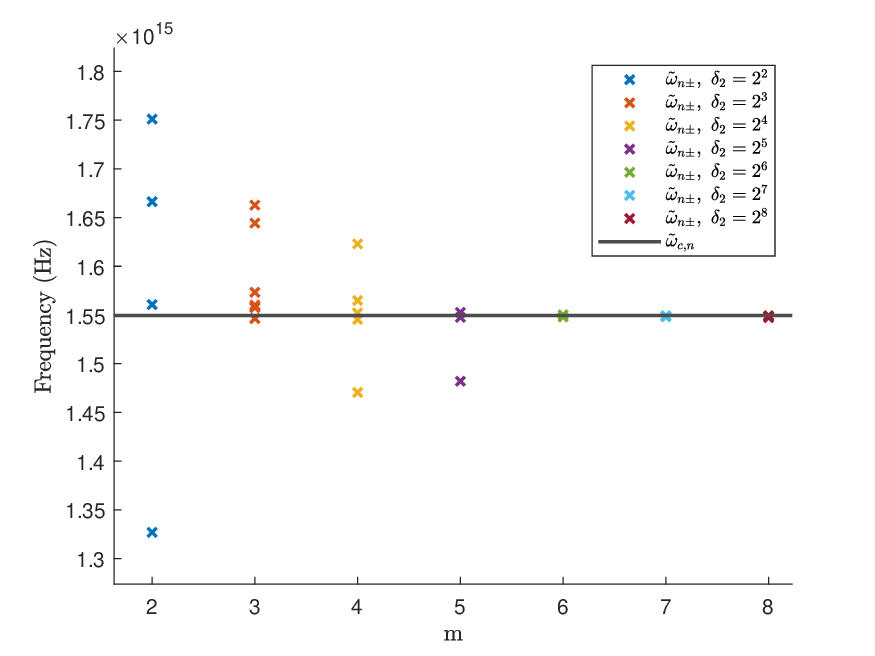}
\end{minipage}}
\subfigure[]{
\begin{minipage}[b]{0.48\textwidth}
\includegraphics[width=1\linewidth]{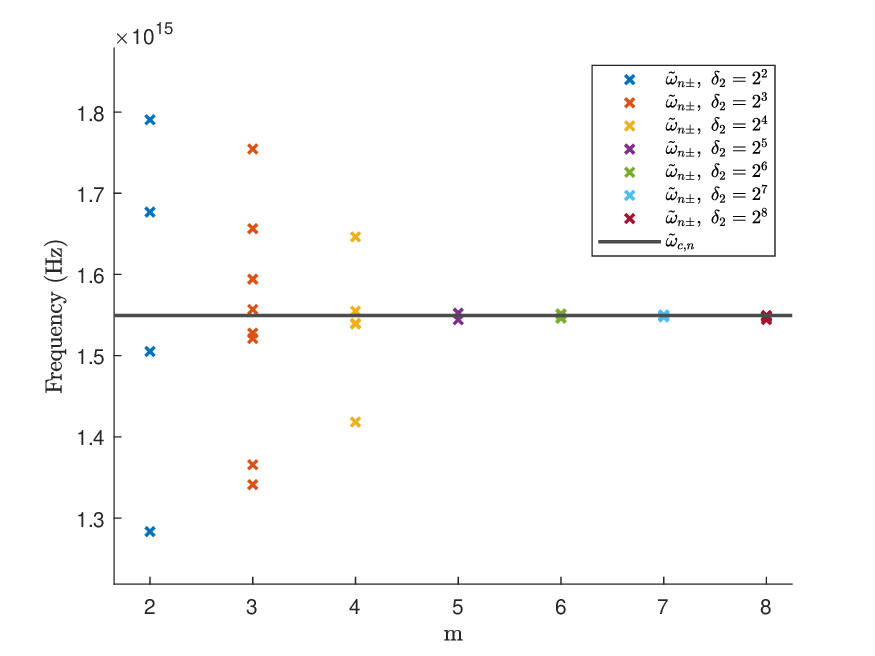}
\end{minipage}}
\caption{Hybridization intensity of solid-cavity plasmon depending on the scale factor $\delta_2$ with different perturbation parameters (a) $\epsilon_1 = \epsilon_2 = 0.01$; (b) $\epsilon_1 = \epsilon_2 = 0.05$; (c) $\epsilon_1 = \epsilon_2 = 0.1$.}
\label{figure_cavity_per}
\end{figure}

\section{Conclusion}

In this paper, we have investigated the plasmonic resonant frequency of a metal nanoshell with a particular geometry and explained in physical terms how this frequency relates to the hybridization of solid and cavity plasmon modes. Our results can be qualitatively and quantitatively applied to guide the design of metal nanoshell structures and predict their resonance properties, thus realizing the value of their applications, and it is validated in numerical experiments. The idea can be extended in several directions: (i) to investigate the situation in three dimensions; (ii) to completely separate the scale factors from the layer potential and discuss nanoparticles with more general geometries; (iii) to generalize the results to elastic systems etc. These new developments will be reported in our forthcoming works.

\section*{Acknowledgements}
\addcontentsline{toc}{section}{Acknowledgments}
The research of H. Liu was supported by NSFC/RGC Joint Research Scheme, N CityU101/21, ANR/RGC
Joint Research Scheme, A-CityU203/19, and the Hong Kong RGC General Research Funds (projects 11311122, 11304224 and 11300821).
The research of Z. Miao was supported by the Hong Kong Scholars Program grant XJ2024057.
The research of G. Zheng was supported by the NSF of China (12271151), NSF of
Hunan (2020JJ4166) and NSF Innovation Platform Open Fund project of Hunan Province (20K030).

\end{document}